\documentclass{amsart}

\usepackage{latexsym}
\usepackage{verbatim}

\usepackage{amssymb,amsmath,amscd,graphicx,color,enumerate}

\usepackage[normalem]{ulem}
\renewcommand\sout{\bgroup\markoverwith
{\textcolor{red}{\rule[0.7ex]{3pt}{1.4pt}}}\ULon}

\usepackage[font=bf]{caption}

\usepackage[all]{xy}

\usepackage{xcolor}

\newcommand\op[1]{\psi^{#1}(M; E)}
\newcommand{\clop}{\overline{\psi^{0}}(M; E)}
\newcommand{\clopn}{\overline{\psi^{-1}}(M; E)}
\newcommand{\Hom}{\operatorname{Hom}}
\newcommand{\End}{\operatorname{End}}
\newcommand{\Aut}{\operatorname{Aut}}
\newcommand{\ind}{\operatorname{ind}}

\newcommand{\Ind}{\operatorname{Ind}}
\newcommand{\Prim}{\operatorname{Prim}}

\newcommand{\CC}{\mathbb C}

\newcommand{\RR}{\mathbb R}

\newcommand{\ZZ}{\mathbb Z}

\newcommand{\maC}{\mathcal C}

\newcommand{\maH}{\mathcal H}
\newcommand{\maK}{\mathcal K}
\newcommand{\maL}{\mathcal L}

\newcommand{\maR}{\mathcal R}

\newcommand\pa{\partial}

\newcommand\ede{\, := \,}
\newcommand\seq{\, = \,}

\newtheorem{theorem}{Theorem}[section]
\newtheorem{lemma}[theorem]{Lemma}
\newtheorem{proposition}[theorem]{Proposition}
\newtheorem{corollary}[theorem]{Corollary}

\theoremstyle{definition}
\newtheorem{definition}[theorem]{Definition}
\newtheorem{remark}[theorem]{Remark}

\begin{document}

\title[Fredholm conditions]{Fredholm conditions for invariant
  operators: finite abelian groups and boundary value problems}

\author[A. Baldare]{Alexandre Baldare}
\email{alexandre.baldare@univ-lorraine.fr}

\author[R. C\^ome]{R\'emi C\^ome} \email{remi.come@univ-lorraine.fr}

\author[M. Lesch]{Matthias Lesch} \address{M.L.: Mathematisches
  Institut, Universit\"at Bonn, Endenicher Allee 60, 53115 Bonn,
  Germany} \email{ml@matthiaslesch.de, lesch@math.uni-bonn.de}
\urladdr{www.matthiaslesch.de, www.math.uni-bonn.de/people/lesch}
\thanks{M.L. was partially supported by the Hausdorff Center for
  Mathematics, Bonn. }

\author[V. Nistor]{Victor Nistor} \email{nistor@math.psu.edu}
\address{A.B, R.C., and V.N. Universit\'{e} Lorraine, 57000 Metz,
  France} \urladdr{http://www.iecl.univ-lorraine.fr/~Victor.Nistor}

\thanks{A.B., R.C., and V.N. have been partially supported by
  ANR-14-CE25-0012-01 (SINGSTAR). Manuscripts available from {\bf
    http:{\scriptsize//}www.math.psu.edu{\scriptsize/}nistor{\scriptsize/}.
} }

\dedicatory{We dedicate this paper to Professor Dan Voiculescu on the
  occasion of his 70th birthday}

\begin{abstract}
We answer the question of when an invariant pseudodifferential
operator is Fredholm on a fixed, given isotypical component. More
precisely, let $\Gamma$ be a compact group acting on a smooth,
compact, manifold $M$ without boundary and let $P \in \psi^m(M; E_0,
E_1)$ be a $\Gamma$-invariant, classical, pseudodifferential operator
acting between sections of two $\Gamma$-equivariant vector bundles
$E_0$ and $E_1$. Let $\alpha$ be an irreducible representation of the
group $\Gamma$. Then $P$ induces by restriction a map $\pi_\alpha(P) :
H^s(M; E_0)_\alpha \to H^{s-m}(M; E_1)_\alpha$ between the
$\alpha$-isotypical components of the corresponding Sobolev spaces of
sections. We study in this paper conditions on the map $\pi_\alpha(P)$
to be Fredholm. It turns out that the discrete and non-discret cases
are quite different. Additionally, the discrete abelian case, which
provides some of the most interesting applications, presents some
special features and is much easier than the general case. Moreover,
some results are true only in the abelian case. We thus concentrate in
this paper on the case when $\Gamma$ is {\em finite abelian}. We prove
then that the restriction $\pi_\alpha(P)$ is Fredholm if, and only if,
$P$ is ``$\alpha$-elliptic,'' a condition defined in terms of the
principal symbol of $P$ (Definition \ref{def.chi.ps}). If $P$ is
elliptic, then $P$ is also $\alpha$-elliptic, but the converse is not
true in general. However, if $\Gamma$ acts freely on a dense open
subset of $M$, then $P$ is $\alpha$-elliptic for the given fixed
$\alpha$ if, and only if, it is elliptic. The proofs are based on the
study of the structure of the algebra $\psi^{m}(M; E)^\Gamma$ of
classical, $\Gamma$-invariant pseudodifferential operators acting on
sections of the vector bundle $E \to M$ and of the structure of its
restrictions to the isotypical components of $\Gamma$. These
structures are described in terms of the isotropy groups of the action
of the group $\Gamma$ on $E \to M$.
\end{abstract}

\maketitle \tableofcontents

\section{Introduction}

Fredholm operators have many applications to Mathematical Physics, to
Partial Differential Equations (linear and non-linear), to Geometry,
and to other areas of mathematics. They have their origin in the work
of several mathematicians on spectral theory and on integral equations
at the end of the nineteenths century. Fredholm operators are
ubiquitous in applications since on a compact manifold, a classical
pseudodifferential operator is Fredholm (acting between suitable
Sobolev spaces) if, and only if, it is elliptic.

In this paper, we obtain {\em an analogous result for the restriction
  of an invariant, classical pseudodifferential operator to a fixed
  isotypical component for the action of a finite abelian group
  $\Gamma$.} Of course, a $\Gamma$-invariant operator is Fredholm if,
and only if, its restrictions to {\em all} isotypical components are
Fredholm.  Our result, focuses on one {\em fixed} given isotypical
component. Namely, the restriction to the isotypical component
corresponding to an irreducible representation $\alpha$ of $\Gamma$ is
Fredholm if, and only if, the operator is $\alpha$-elliptic
(Definition \ref{def.chi.ps} and Theorem \ref{thm.main1} below). The
reasons we assume our group to be abelian and discrete are, first,
that the main result is nolonger correct as stated in the non-discrete
case and, second, that the general discrete case is quite different
(and much more difficult) than the abelian discrete case. Moreover,
some useful intermediate results are true only in the abelian case and
this is the case needed for applications to boundary value problems.

Although in the formulation of our main result we do not use
$C^*$-algebras, for its proof, we have found it convenient to use
them. Recently, there were quite a few papers using $C^*$-algebras to
obtain Fredholm conditions, see, for instance, \cite{dMG, DS1, LMN,
  LMR} among many others. Often groupoids were also used \cite{CNQ,
  DS2, Mo1, Re}. Fredholm conditions play an important role in Quantum
Mechanics in the study of the essential spectrum of $N$-body
Hamiltonians \cite{BLLS1, GI, Ge, HM, LS}. A powerful related
technique is that of ``limit operators'' \cite{LR, Li, LiS, RRS}. Some
of the most recent papers using similar ideas include \cite{Zhang,
  CCQ, CNQ, Remi, MougelH, Ma2, MaNi, vEY2}, to which we refer for
further references. Besides $C^*$-algebras, pseudodifferential
operators were also often used to obtain Fredholm conditions, see
\cite{DLR, LauterMoroianu1, LeschVertman} and the references therein.

Let us now explain our main result in more detail.

\subsection{$\Gamma$-principal symbol and $\alpha$-ellipticity}
In general, for any compact group $G$, we let $\widehat G$ denote the
set of equivalence classes of irreducible $G$-modules (or
representations), as usual. It is a finite set if $G$ is finite. If $T
: V \to W$ is a $G$-equivariant map of $G$-modules and $\alpha \in
\widehat G$, we let
\begin{equation}\label{eq.restriction}
    \pi_\alpha(T) : V_{\alpha} \to W_{\alpha}
\end{equation}
denote by the induced map obtained by restricting and corestricting
$T$ to the corresponding isotypical component. Our main result,
  Theorem \ref{thm.main1} is stated in terms of the $\Gamma$-principal
  symbol and $\alpha$-ellipticity, two concepts which we now
  introduce.

Let $\Gamma$ be a {\em finite abelian} group acting on a smooth,
compact, boundaryless manifold $M$ and let $P \in \psi^m(M; E_0, E_1)$
be a $\Gamma$-invariant classical pseudodifferential operator acting between
sections of two $\Gamma$-equivariant vector bundles $E_0$ and
$E_1$. Let $\alpha \in \widehat{\Gamma}$ and consider as above
\begin{equation}\label{eq.def.Pchi}
  \pi_\alpha(P) \, : \, H^s(M; E_0)_\alpha \, \to \, H^{s-m}(M;
  E_1)_\alpha\,,
\end{equation}
acting between the $\alpha$-isotypical components of the corresponding
Sobolev spaces of sections. The main question that we answer in this
paper is to determine when $\pi_\alpha(P)$ is Fredholm in terms of its
principal symbol
\begin{equation}\label{eq.princ.symb}
  \sigma_m(P) \in \Gamma(T^*M \smallsetminus \{0\}; \Hom(E_0, E_1))\,,
\end{equation}
regarded as a homogeneous function on the cotangent bundle of $M$.
Note that for $\Gamma$ abelian, $\widehat{\Gamma}$ consists of
characters, that is, of group morphisms $\Gamma \to \CC^*$. The reason
for restricting to the case $\Gamma$ finite abelian is that it
presents some special features, although some, but not all, of our
results extend to the case $\Gamma$ finite (possibly non-abelian) and
even to the case $\Gamma$ compact. The case $\Gamma$ abelian provides
some of the most important applications and is much easier than the
general case, so, for the sake of the clarity and brevity of the
presentation, we will assume in our main result that $\Gamma$ is
abelian. Moreover, some of our results are not true in the non-abelian
case, in general, and the main result (Theorem \ref{thm.main1}) is not
true for non-discrete groups (Corollary 2.5 of
\cite{atiyahGelliptic}).

For simplicity, we will consider only classical pseudodifferential
operators in this article \cite{LeschBook, Taylor2, Treves}. Recall
that a classical pseudodifferential operator $P$ is called {\em
  elliptic} if its principal symbol is invertible (away from the zero
section). If $P$ is elliptic, then it is Fredholm, and hence
$\pi_\alpha(P)$ is also Fredholm. The converse is not true, however,
in general. Indeed, we introduce, for any irreducible representation
$\alpha$ of $\Gamma$, an ``$\alpha$-principal symbol''
$\sigma_m^\alpha(P)$ (Definition \ref{def.chi.ps}) and prove that
$\pi_\alpha(P)$ is Fredholm if, and only if, its $\alpha$-principal
symbol is invertible (in which case we call $P$ {\em
  $\alpha$-elliptic}, see Theorem \ref{thm.main1} below for the
precise statement). As we have just noticed, in general, the
invertibility of the $\alpha$-principal symbol does not imply
ellipticity. To state these results in more detail, we need to
introduce some notation and terminology.

The $\Gamma$-invariance of $P$ implies that its principal symbol is
also $\Gamma$ invariant:
\begin{equation*}
  \sigma_m(P) \in \Gamma(T^*M \smallsetminus
  \{0\}; \Hom(E_0, E_1))^\Gamma \,.
\end{equation*}
To study the space of morphisms $\Gamma(T^*M \smallsetminus \{0\};
\Hom(E_0, E_1))^\Gamma$ in which the principal symbol $\sigma_m(P)$
lives, let
\begin{equation}\label{eq.def.Gamma.xi}
   \Gamma_\xi \ede \{ \gamma \in \Gamma \, \vert \ \gamma \xi =
   \xi \}
\end{equation}
denote the isotropy group of a $\xi \in T_x^*M$ in $\Gamma$, $x \in
M$, as usual. The isotropy $\Gamma_x$ of $x \in M$ is defined
similarly. Then the groups $\Gamma_\xi \subset \Gamma_x$ act on
$E_{0x}$ and on $E_{1x}$, the fibers of $E_0, E_1 \to M$ at $x$. If $q
\in \Gamma(T^*M \smallsetminus \{0\}; \Hom(E_0, E_1))^\Gamma$, then
$q(\xi) \in \Hom(E_{0x}, E_{1x})^{\Gamma_\xi}$. As we will see below,
there is no loss of generality for our main result to assume that $E_0
= E_1 = E$, in which case $\Hom(E_0, E_1) = \End(E)$.

Let us consider the space 
\begin{equation}\label{eq.def.XMG}
  X_{M, E, \Gamma} \ede \{ ( \xi, \rho) \, \vert \ \xi \in T^*M
  \smallsetminus \{0\}, \ \rho \in \widehat \Gamma_\xi, \mbox{ and }
   \Hom(\rho, E_x)^{\Gamma_\xi} \neq 0 \} \,.
\end{equation}
Let also $\rho \in \widehat{\Gamma}_\xi$ be an irreducible
representation of $\Gamma_\xi$, then
\begin{equation}\label{eq.hatQ}
  \hat{q} (\xi, \rho) \ede \pi_{\rho}(q(\xi)) \in \End(E_{x
    \rho})^{\Gamma_\xi}
\end{equation}
denotes the restriction of $q(\xi)$ to the isotypical component
corresponding to $\rho$, with $\pi_\rho$ defined in Equation
\eqref{eq.restriction}.  Thus $\hat{q}$ is a function on $X_{M, E,
  \Gamma}$. Applying this construction to $\sigma_m(P) \in \Gamma(T^*M
\smallsetminus \{0\}; \End(E))^\Gamma$, we obtain the function
\begin{equation}\label{eq.def.Gamma.symb}
  \sigma_m^\Gamma(P) \ede \widehat{\sigma_m(P)} : X_{M,E, \Gamma} \to
  \bigcup_{(x, \rho) \in X_{M,E, \Gamma}} \End(E_{x
    \rho})^{\Gamma_\xi}\,.
\end{equation}
That is,
\begin{equation}
  \sigma_m^\Gamma (P) (\xi, \rho) \ede \pi_{\rho}(\sigma_m(P)(\xi))
  \in \End(E_{x \rho})^{\Gamma_\xi}\,, \ \ \xi \in T_x^*M\,.
\end{equation}

The characterization of Fredholm operators can be reduced to each
component of the manifold. We shall therefore assume for the
  rest of this Introduction and beginning with Subsection
  \ref{ssec.isotropy} that our manifold $M$ is connected. This
  simplifies also the statements and the proof of our results. Let
  then $\Gamma_0$ be a minimal isotropy group for the connected
  manifold $M$ (see Subsection \ref{ssec.principal}). 

The $\alpha$-principal symbol $\sigma_m^\alpha(P)$ of $P$, $\alpha \in
\widehat{\Gamma}$, is defined in terms of $\sigma_m^\Gamma(P)$, but in
order to define it, we need an additional crucial ingredient that
takes $\alpha$ into account. For $\alpha \in \widehat{\Gamma}$ and
$\rho \in \widehat{\Gamma}_\xi$, we will say that $\alpha$ and $\rho$
are {\em $\Gamma_0$-disjoint} if $\Hom_{\Gamma_0} (\rho, \alpha) = 0$,
otherwise, we will say that they are {\em $\Gamma_0$-associated}. Let
\begin{equation}\label{eq.def.Xalpha}
   X^{\alpha}_{M, E, \Gamma} \ede \{(\zeta, \rho) \in X_{M, E ,
     \Gamma} \, \vert \ \rho \mbox{ and } \alpha \mbox{ are
     $\Gamma_0$-associated } \} \,.
\end{equation}

Let us assume for the rest of this introduction that $\Gamma$ is {\em
  abelian}. Then we have that $\alpha \in \widehat{\Gamma}$ and $\rho
\in \widehat{\Gamma}_\xi$ are $\Gamma_0$-associated if, and only if,
their restrictions to $\Gamma_0$ coincide, that is, if $\alpha
\vert_{\Gamma_0} = \rho \vert_{\Gamma_0}$. We can now define the {\em
  $\alpha$-principal symbol} $\sigma_m^\alpha (P)$ of $P$.

\begin{definition}\label{def.chi.ps}
The {\em $\alpha$-principal symbol} $\sigma_m^\alpha (P)$ of $P$ is
the restriction of $\sigma_m^\Gamma(P)$ to $X^\alpha_{M, E, \Gamma}$:
\begin{equation*}
  \sigma_m^\alpha(P) \ede \sigma_m^\Gamma(P)\vert_{X^\alpha_{M, E,
      \Gamma}}\,.
\end{equation*}
We shall say that $P \in \psi^m(M; E)$ is {\em $\alpha$-elliptic} if
its $\alpha$-principal symbol $\sigma_m^\alpha(P)$ is invertible
everywhere on its domain of definition. This definition extends right
away to operators in $\psi^m(M; E_0, E_1)$.
\end{definition}

\subsection{Statement of the main result}
Thus $P$ is $\alpha$-elliptic if, and only if, $\sigma_m^\Gamma(P)$ is
invertible on $X_{M, E, \Gamma}^\alpha$. We are ready now to state our main
result.

\begin{theorem}\label{thm.main1}
  Let $\Gamma$ be a finite \emph{abelian} group acting on a smooth,
compact manifold $M$ and let $P \in \psi^m(M; E_0, E_1)$ be a
classical pseudodifferential operator acting between sections of two
$\Gamma$-equivariant bundles $E_0, E_1 \to M$, $m \in \RR$, and
$\alpha \in \widehat{\Gamma}$. We have that $\pi_\alpha(P) : H^s(M;
E_0)_\alpha \, \to \, H^{s-m}(M; E_1)_\alpha$ is Fredholm if, and only
if $P$ is $\alpha$-elliptic.
\end{theorem}

If $\Gamma$ acts without fixed points on a {\em dense} open subset of
$M$, then $\Gamma_0 = 1$, and hence $X_{M, E, \Gamma} = X^\alpha_{M,
  E, \Gamma}$ for all $\alpha \in \widehat{\Gamma}$. Hence, in this
case, $P$ is $\alpha$-elliptic if, and only if, it is elliptic. The
ellipticity of $P$ can thus be checked in this case simply by looking
at a single isotypical component. We stress, however, that if $\Gamma$
is not discrete, this statement, as well as the statement of our main
result (Theorem \ref{thm.main1} above), are not true anymore. However,
our main result, as well as many intermediate results hold for general
finite groups and some even for compact Lie groups. The
  extension of our main result to general finite groups is work in
  progress \cite{BCLN}, but the proof seems to be much more involved.

A motivation for our result comes from index theory. Let us assume
that $P$ is $\Gamma$-invariant and elliptic. Atiyah and Singer have
determined, for any $\gamma \in \Gamma$, the value at $\gamma$ of the
character of $\ind_\Gamma(P) \in R(G)$, that is they have computed
$ch_\gamma(\ind_\Gamma(P)) \in \CC$ in terms of data at the fixed
points of $\gamma$ on $M$ \cite{AS3}. (Here $R(G) :=
\ZZ^{\widehat{G}}$ is the representation ring of $G$.)

Br\"uning \cite{Bruning78, Bruning88} considered the "isotypical heat
trace" $\mathrm{tr}( p_\alpha e^{-t\Delta} )$, which is nothing but
the heat trace of $\pi_\alpha(\Delta)$, and its short time asymptotic
expansion. Clearly, carrying out Br\"uning's programme in full would
lead to a heat equation proof of an index theorem for the $\alpha$
isotypical component of Dirac type operators.  However, the technical
obstacles for this approach are enormous.

\subsection{Contents of the paper}
\label{sub:Contents of the paper}
Let us quickly describe here the contents of our paper. We start in
Section \ref{sec.Preliminaries} with some preliminaries. We thus
recall some facts about groups, most notably Frobenius reciprocity
(for finite groups) and the definitions of induced representations, of
minimal isotrypy groups (for connected $M$) and of the principal orbit
bundle. We also review some notions concerning the primitive ideal
spectrum of $C^*$-algebras, as well as basic facts concerning
(equivariant) pseudodifferential operators.

In Section \ref{sec.structure.of.regularizing.operators}, we compute
the image of the algebra $\overline{\psi^{-1}}(M;E)$ of regularizing
operators via $\pi_\alpha$. We do this by proving some general results
on the structure of $C^*$-algebras with an inner action of our group
$\Gamma$. When the action of the group $\Gamma$ is inner, the results
and their proofs become simpler.

Let $A_M := \maC_0(M; \End(E))$. The main difficulties
arise in Section \ref{sec.principal.symbol}. There, we identify the
primitive spectrum of the $C^*$-algebra $A_M^\Gamma$ of
$\Gamma$-invariant symbols with the set $X_{M,E,\Gamma}/\Gamma$
described above. Some care is taken to describe the corresponding
topology on $X_{M,E,\Gamma}/\Gamma$. We then consider the projection
from $A_M^\Gamma$ to the Calkin algebra of $L^2(M;E)_\alpha$ and show
that the closed subset of $\Prim A_M^\Gamma$ associated to its kernel
is $X_{M,E,\Gamma}^\alpha/\Gamma$. These results are
used in Section \ref{sec.applications} to prove the main result of the
paper, Theorem \ref{thm.main1}. We also discuss an application to
mixed boundary value problems and explain why our result is not true
when the group $\Gamma$ is not discrete,

The last named author (V.N.) thanks Max Planck Institute for support
while this research was performed. Since this paper is dedicated to
Professor Dan Voiculescu, the last named author would like to mention
that his papers, most notably \cite{PV1, PV2, VoiculescuCR}, and
\cite{Voiculescu} have played an important role in this author's
formation as a mathematician in his early years. Moreover, we are
happy to dedicate to Voiculescu this paper in which we prove a result
that does not explicitly use $C^*$-algebras, but whose proof uses in
an essential way the theory of these algebras. This was the spirit of
interdisciplinarity that Voiculescu was promoting while he was a
member of INCREST, an institute that is now called the Institute of
the Roumanian Academy of Sciences (IMAR).

\section{Preliminaries}
\label{sec.Preliminaries}

We begin by setting up the terminology and the notation used in this
paper. We also recall some basic results that are needed in the
sequel.  

Throughout the paper, $\Gamma$ will be a compact group acting on a
locally compact space $M$. For the most part, $M$ will be a smooth
Riemannian manifold and $\Gamma$ will be a {\em compact Lie group}
acting smoothly and isometrically on $M$. The final result holds only
for discrete (thus finite) groups and $M$ compact, but many
intermediate results hold in greater generality, so we have tried to
state the results in the greatest generality possible when this did
not involve too much extra work. In particular, we shall start with a
compact group $\Gamma$ acting on a possibly non-compact Riemannian
manifold\footnote{The metric on $M$ is, however, for convenience
  only.} $M$.  Eventually, we shall assume that $M$ is compact and
that $\Gamma$ is discrete, hence finite. Moreover, since the case
$\Gamma$ abelian is simpler and presents some special features, for
the final result we will assume that $\Gamma$ is abelian.

\subsection{Group representations}
\label{ssec.group.actions}
We follow the standard conventions, see \cite{tomDieckRepBook,
  SerreBook}, to which we refer for further references and unexplained
concepts.

\subsubsection{Group actions on sets}
Let us assume now that $\Gamma$ is a group that acts on a set $M$.  If
$x \in M$, then $\Gamma x$ denotes the $\Gamma$ {\em orbit} of $x$ and
$\Gamma_x$ denotes the {\em isotropy} group (of $M$) at $x$, that is
\begin{equation}
   \Gamma_x := \{ \gamma \in \Gamma \, \vert \ g x = x \} \subset
   \Gamma\,.
\end{equation}
We shall write $H \sim H'$ if the subgroups $H$ and $H'$ are
conjugated in $\Gamma$. If $H \subset \Gamma$ is a subgroup, we shall
denote by $M_{(H)}$ the set of elements of $M$ whose isotropy
$\Gamma_{x}$ is conjugated to $H$ (in $\Gamma$), i.e. $H \sim
\Gamma_x$. 

\subsubsection{Representations and isotypical components}
\label{ssec.representations}
Let $V$ be a locally convex space and $\maL(V)$ denote the set of {\em
  continuous} linear maps $V \to V$. Let $\Gamma$ now be a {\em
  compact topological group} and $\rho : \Gamma \to \maL(\maH_\rho)$
be a {\em strongly continuous} representation in a \emph{complete,
  locally convex topological} vector space $\maH_\rho$, in the sense
that, for each $\xi \in \maH_\rho$, the map $\Gamma \ni \gamma \to
\rho(\gamma) \xi \in \maH_\rho$ is continuous. We shall say then that
$\maH_\rho$ is a $\Gamma$-module and we shall often drop $\rho$ from
the notation, thus $\maH = \maH_\rho$ and $\gamma \xi := \rho(\gamma)
\xi$. If $\Gamma$ is a discrete group (as it will be the case for our
final results), the continuity conditions will, of course, be
automatically satisfied.

For any two $\Gamma$-modules $\maH$ and $\maH_1$, we shall denote by
\begin{equation*}
    \Hom_{\Gamma}(\maH, \maH_1) \seq \Hom(\maH, \maH_1)^\Gamma \seq
    \maL(\maH, \maH_1)^\Gamma
\end{equation*}
the set of continuous linear maps $T : \maH \to \maH_1$ that commute
with the action of $\Gamma$, that is, $T (\gamma \xi) = \gamma T(\xi)$
for all $\xi \in \maH$ and $\gamma \in \Gamma$. 

Finally, we denote by $\hat \Gamma$ the set of equivalence classes of
irreducible unitary $\Gamma$-modules. We shall need the following
terminology.

\begin{definition}\label{def.associated}
Let $H \subset \Gamma_1$ and $H \subset \Gamma_2$ be compact
groups. Two irreducible representations $\alpha_i \in
\widehat{\Gamma}_i$, $i = 1,2$, are called {\em $H$-associated} if
$\maL(\alpha_1, \alpha_2)^H \neq 0$. Otherwise, we shall
say that they are {\em $H$-disjoint}.
\end{definition}

If $\Gamma_i$, $i=1,2$, are both abelian, then the irreducible
representations $\alpha_i$ are characters, that is, morphisms
$\alpha_i : \Gamma_i \to \CC^*$, and we have that they are associated
if, and only if, $\alpha_1\vert_{H} = \alpha_2\vert_{H}$.

\subsubsection{Isotypical component}
\label{ssec.isotypical_component}

Let $\maH$ be a $\Gamma$-module and $\alpha \in \hat{\Gamma}$. Then
$p_\alpha$ will denote the $\Gamma$-invariant projection onto the
$\alpha$-isotypical component $\maH_\alpha$ of $\maH$, defined as the
largest (closed) $\Gamma$ submodule of $\maH$ that is isomorphic to a
multiple of $\alpha$. In other words, $\maH_\alpha$ is the sum of all
$\Gamma$-submodules of $\maH$ that are isomorphic to
$\alpha$. Equivalently, since $\Gamma$ is compact, we have
$\maH_{\alpha} \simeq \alpha \otimes \Hom_{\Gamma}(\alpha, \maH)$;
moreover
\begin{equation}\label{eq.isotypical}
  \maH_\alpha \, \neq \, 0 \ \Leftrightarrow \ \Hom_\Gamma(\alpha,
  \maH) \, \neq \, 0 \ \Leftrightarrow \ \Hom_\Gamma(\maH, \alpha) \,
  \neq \, 0.
\end{equation}
If $T \in \maL(\maH)$ commutes with the action of $\Gamma$ (i.e. it is
a $\Gamma$-module morphism), then $T(\maH_\alpha) \subset \maH_\alpha$
and we denote by
\begin{equation} \label{eq.restriction2}
    \pi_\alpha : \maL(\maH)^\Gamma \to \maL(\maH_\alpha) \,, \quad
    \pi_\alpha(T) \ede T\vert_{\maH_\alpha}\,,
\end{equation}
the associated morphism and operator, as in Equation
\eqref{eq.restriction} of the Introduction. The morphism $\pi_\alpha$
will play a crucial role in what follows.

\subsubsection{Convolution algebras}
The algebra $\maC(\Gamma)$ of continuous functions on $\Gamma$,
endowed with the convolution product, will act on any $\Gamma$-module
$\maH$. If, moreover, $\maH$ is a Hilbert space and $\Gamma$ acts by
unitary operators, we shall say then that $\maH$ is a {\em unitary
  $\Gamma$-module.}  We endow $\Gamma$ with a fixed Haar measure and
let $C_r^*(\Gamma)$ be the completion of $\maC(\Gamma)$ acting on
$L^2(\Gamma)$, which is a unitary $\Gamma$-module. Then
$C_r^*(\Gamma)$ will act on any unitary $\Gamma$-module $\maH$, since
$\Gamma$ is compact (and hence amenable: $C_r^*(\Gamma) =
C^*(\Gamma)$).

For any algebra $A$, we shall denote by $Z(A)$ the {\em center} of
$A$, that is, the set of elements $z \in A$ that commute with all
other elements $a \in A$. An element $z \in Z(A)$ will be called {\em
  central} (in $A$). For instance, if $\alpha \in \widehat \Gamma$,
then its character defines a central projection $p_\alpha \in
Z(\maC(\Gamma)) \subset Z(C^*(\Gamma))$. More explicitely, for any
$\gamma \in \Gamma$ we have that
\begin{equation*}
  p_\alpha(\gamma) = \chi_\alpha(\gamma) \ede \mathrm{tr}(\alpha(\gamma)).
\end{equation*}
Given a representation $\rho$ of $\Gamma$ on $\maH$, the image of
$p_\alpha$ is then
\begin{equation*}
  \rho(p_\alpha) = \int_\Gamma \chi_\alpha(g) \rho(g) dg,
\end{equation*}
where integration is against the Haar measure.  We are interested in
this projection since $\maH_\alpha = p_\alpha \maH$.

\subsection{Induction and Frobenius reciprocity}
\label{ssec.Frobenius}
We now review some basic definitions and results for induced
representations. We will use induction {\em for finite groups only,}
so we assume in this discussion of the Frobenius reciprocity (i.e. in
this subsection) that $\Gamma$ is finite.

\subsubsection{Definition of the induced module}
Since we are assuming in this subsection that $\Gamma$ is finite, we
have that $C^*(\Gamma) = \maC(\Gamma) = \CC[\Gamma]$, the group
algebra of $\Gamma$. We will use the standard notation $V^{(I)} := \{
f : I \to V \}$, valid for $I$ finite. If $H \subset \Gamma$ is a
subgroup (hence also finite) and $V$ is a $H$-module, we let
\begin{equation}\label{eq.def.induced}
  \Ind_H^\Gamma (V) \ede \CC[\Gamma] \otimes_{\CC[H]} V \, \simeq \,
  \{\, \xi : \Gamma \to V \, |\ f(gh^{-1}) = h f(g) \, \} \, \simeq \,
  V^{(\Gamma/H)}\,
\end{equation}
be the {\em induced representation} from $V$. The last isomorphism is
obtained by choosing a set of representatives of the right cosets
$\Gamma/H$. The action of $\Gamma$ on $\Ind_H^\Gamma (V)$ is by left
multiplication on $\CC[\Gamma]$, and the indicated isomorphism is an
isomorphism of $\Gamma$-modules. The $\Gamma$-module
$\Ind_H^\Gamma(V)$ depends functorially on $V$.

\begin{remark}\label{rem.CD}
If $V$ is an algebra and the group $H$ acts on $V$ by algebra
homomorphisms, then the isomorphism $\Ind_H^\Gamma (V) \simeq \{\xi :
\Gamma \to V \, |\ f(gh^{-1}) = h f(g)\}$ shows that
$\Ind_H^\Gamma(V)$ is an algebra for the pointwise product. If $V_1$
is a left $V$-module (with a structure compatible with the action of
$\Gamma$), then $\Ind_H^\Gamma (V_1)$ is a $\Ind_H^\Gamma (V)$ module,
again with the pointwise multiplication. The induction is moreover
compatible with morphisms of modules and algebras (change of scalars),
again by the function representation of the induced representation. In
particular, if $\phi : V \to W$ is a $H$-morphism of algebras, if
$V_1$ and $W_1$ are modules over these algebras, and $\psi : V_1 \to
W_1$ is a $H$-module morphism such that $\psi( a b) = \phi(a)
\psi(b)$, then, if $mult$ denotes the multiplication map, the
following diagram commutes:
\begin{equation}\label{eq.CD}
\begin{CD}
  \ind_H^\Gamma(V) \otimes \ind_H^\Gamma(V_1) @>{\ \phi \otimes \psi
    \ }>> \ind_H^\Gamma(W) \otimes \ind_H^\Gamma(W_1) \\
  @V{mult}VV
  @VV{mult}V \\
  \ind_H^\Gamma(V_1) @>{ \psi } >> \ind_H^\Gamma(W_1) \,.
\end{CD}
\end{equation}
(All maps, including the multiplications, are assumed to be compatible
with the action of $\Gamma$.)
\end{remark}

\subsubsection{Explicit isomorphisms}
We will use the following form of the {\em Frobenius reciprocity}: the map
\begin{equation}\label{eq.Frobenius}
  \begin{gathered}
    \Phi = \Phi_{H, V}^{\Gamma, \maH} : \Hom_{H}( \maH, V ) \, \to \,
    \Hom_{\Gamma}( \maH, \Ind_H^\Gamma(V) ) \,,\\  \Phi(f)(\xi)
      \ede \frac{1}{|H|} \, \sum_{g \in \Gamma} g \otimes_{\CC[H]}
      f(g^{-1} \xi)\,,
  \end{gathered}
\end{equation}
is an isomorphism ($\xi \in \maH$, $f \in \Hom_{H}( \maH, V )$). This
version of the Frobenius reciprocity is not valid in general, but is
valid for finite groups \cite{tomDieckRepBook, SerreBook}. Often one
writes $\Hom_{H}( \operatorname{Res_H^\Gamma}(\maH), V )$ instead of
$\Hom_{H}( \maH, V )$. Let $\alpha$ be an irreducible representation
of $\Gamma$, $H \subset \Gamma$ be a subgroup, and $\beta$ an
irreducible representation of $H$. Frobenius reciprocity gives, in
particular, that the multiplicity of $\alpha$ in
$\ind_H^\Gamma(\beta)$ is the same as the multiplicity of $\beta$ in
the restriction of $\alpha$ to $H$. In particular, $\alpha$ is
contained in $\ind_H^\Gamma(\beta)$ if, and only if, $\beta$ is
contained in the restriction of $\alpha$ to $H$ i.e.\ $\alpha$
  and $\beta$ are {\em $H$-associated}.  
Furthemore, by taking $\maH$ to be the trivial $\Gamma$-module $\CC$,
we obtain an isomorphism
\begin{equation}\label{eq.Frobenius2}
  \begin{gathered}
    \Phi : V^{H} \seq \Hom_{H}( \CC, V ) \, \simeq\, \Hom_{\Gamma}(
    \CC, \Ind_H^\Gamma(V) ) \seq \Ind_H^\Gamma(V)^\Gamma \,,\\
    \Phi(\xi) \ede \frac{1}{|H|} \, \sum_{g \in \Gamma} g
    \otimes_{\CC[H]} \xi \seq \sum_{x \in \Gamma/H} x \otimes \xi\,.
  \end{gathered}
\end{equation}
The chosen normalization in the definition of $\Phi$ is such that it
is an isomorphism of algebras if $V$ is an algebra.

\subsubsection{The abelian case}
 If $\Gamma$ is abelian and $V$ is an irreducible
$H$-module, then the action of $H$ on $V$ is via scalars: $h \cdot v =
\chi_V(h)v$, for some group morphism (i.e. character) $\chi_V : H \to
\CC^*$. The induced module $\Ind_H^\Gamma (V)$ splits then as the
direct sum of the irreducible $\Gamma$ modules $\beta$ that are
$H$-associated to $V$. If $\chi$ is a character of $\Gamma$, we shall
denote by $V_\chi$ the $H$-module equal to $\CC$ as a vector space, with the
action of $h \in H$ given by $h\cdot v = \chi(h)v$.

\begin{lemma}\label{lemma.abelian.induction}
Assume that $\Gamma$ is a finite abelian group and that $H$ is a
subgroup of $\Gamma$. Let $V$ be an irreducible $H$-module
corresponding to the character $\chi_V : H \to \CC^*$. Then
\begin{equation*}
  \Ind_H^\Gamma(V) \, \simeq \, \bigoplus_{\substack{\chi \in
      \hat\Gamma, \\ \chi_{|H} = \chi_V}} \Ind_H^\Gamma(V)_\chi.
\end{equation*}
Moreover, by writing $\Ind_H^\Gamma(V) \simeq \CC[\Gamma/H] \otimes V$
as vector spaces, we obtain an action of $\widehat{\Gamma/H}$ on
$\Ind_H^\Gamma(V)$ by the formula
\begin{equation*}
   \rho \cdot (\gamma H \otimes v ) \ede \rho(\gamma H) \gamma H
   \otimes v, \quad \gamma \in \Gamma,\ v \in V,\ \mbox{ and }\ \rho
   \in \widehat{\Gamma/H} \,.
\end{equation*}
This action maps $\Ind_H^\Gamma(V)_{\chi}$ to $\Ind_H^\Gamma(V)_{\rho
  \chi}$.
\end{lemma}

\begin{proof}
Given $\chi \in \widehat{\Gamma}$, we have by the Frobenius
isomorphism that $\chi$ is contained in $\Ind_H^\Gamma(V)$ if, and
only if, $\chi \vert_{H} = \chi_V$. If $A$ is a finite abelian group,
we have the (non-canonical) isomorphism $\widehat{A} \simeq A$ of
groups. There are $|\widehat{\Gamma/H}| = |\Gamma/H|$-many characters
$\chi$ of $\Gamma$ with the property $\chi \vert_{H} = \chi_V$. It
follows, by counting dimensions, that they all appear with
multiplicity one in $\Ind_H^\Gamma(V)$. The statement about the action
of $\widehat{\Gamma/H}$ is proved by a direct computation. This proof
is complete.
\end{proof}

\subsubsection{Inducing endomorphism modules}
Let now $\alpha$ be an irreducible $\Gamma$-module, let $\beta_j$, $j
= 1,\ldots, N$, be non-isomorphic irreducible $H$-modules (with $H$ a
subgroup of $\Gamma$, as above), and
\begin{equation}\label{eq.def.beta}
   \beta \ede \oplus_{j=1}^N
   \beta_j^{k_j} \,.
\end{equation}
If $H$ is abelian, then each $\beta_j$ is one dimensional, and
  hence $H$ acts by scalars on each $\beta_j^{k_j}$.

We want to study the algebra $\Ind_H^\Gamma(\End(\beta))^\Gamma$
acting on $\Ind_H^\Gamma(\beta)$ and on its isotypical component
$p_\alpha(\Ind_H^\Gamma(\beta)) = \Ind_H^\Gamma(\beta)_{\alpha}$ (see
\ref{ssec.isotypical_component} for the definition of the projection
$p_\alpha$). We have, by the Frobenius isomorphism and by the form of the
$H$-module $\beta$, that
\begin{equation}\label{eq.direct.sum}
  \Ind_H^\Gamma(\End(\beta))^\Gamma \, \simeq\, \End(\beta)^{H} \,
  \simeq\, \oplus_{j=1}^N \End(\beta_j^{k_j})^H \, \simeq\,
  \oplus_{j=1}^N M_{k_j}(\CC),
\end{equation}
which is a semi-simple algebra. Moreover, we have that
  $\Ind_H^\Gamma(\End(\beta))^\Gamma = \Phi(\End(\beta)^{H})$, where
  $\Phi$ is the map of Equation \eqref{eq.Frobenius2}. From the
properties of the induction functor $\Ind_H^\Gamma$, we also have that
$\Ind_H^\Gamma(\beta) = \oplus_{j=1}^N \Ind_H^\Gamma(\beta_j^{k_j})$.

\begin{lemma}\label{lemma.componentwise}
Let $\beta \ede \oplus_{j=1}^N \beta_j^{k_j}$ be as in Equation
\eqref{eq.def.beta}, let
\begin{equation*}
 T \seq (T_j) \in \End(\beta)^{H}  \, \simeq \, 
   \oplus_{j = 1}^N \End(\beta_j^{k_j})^H\,,
\end{equation*}
with $T_j \in \End(\beta_j^{k_j})^{H}$, and let $\xi_j \in
\Ind_H^\Gamma(\beta_j^{k_j})$. We let
\begin{equation*}
  \xi \ede (\xi_j) \in
  \oplus_{j=1}^N \Ind_H^\Gamma(\beta_j^{k_j}) \simeq
  \Ind_H^\Gamma(\beta) \,.
\end{equation*}
Then $\Phi(T) (\xi) = (\Phi(T_j) \xi_j)_{j=1,\ldots,N}$.
\end{lemma}

In other words, the Frobenius isomorphism $\Phi$ of Equation
\eqref{eq.Frobenius2} is compatible with direct sums and with the
action of morphisms on modules.

\begin{proof}
This follows from the naturality of the product, the isomorphism
\begin{equation}
  \Ind_H^\Gamma(\oplus_{j=1}^N \End(\beta_j^{k_j}))^\Gamma
  \ \stackrel{\sim}{\longrightarrow}\ \Ind_H^\Gamma(\End(\beta))^\Gamma
  \,,
\end{equation}
and Remark \ref{rem.CD} (especially Equation \eqref{eq.CD}).
\end{proof}

Put differently, the simple factor of the algebra
$\Ind_H^\Gamma(\End(\beta))^\Gamma$ corresponding to
$\Ind_H^\Gamma(\End(\beta_j^{k_j}))^\Gamma$ acts only on the $j$th
component of $\oplus_{i=1}^N \Ind_H^\Gamma(\beta_i^{k_i}) =
\Ind_H^\Gamma(\beta)$. The following proposition will play a crucial
role in what follows.

\begin{proposition}\label{prop.res.alpha}
Let $\beta \ede \oplus_{j=1}^N \beta_j^{k_j}$ be as in Equation
\eqref{eq.def.beta}. Let $J \subset \{1, 2, \ldots, N\}$ be the set of
indices $j$ such that $\alpha\vert_{H} = \beta_j$ (i.e. $\alpha$ and
$\beta_j$ are $H$-associated). Then the morphism
\begin{equation*}
    \pi_\alpha : \Ind_H^\Gamma(\End(\beta))^\Gamma \, \simeq\,
    \oplus_{j = 1}^N \Ind_H^\Gamma(\End(\beta_j^{k_j}))^\Gamma \, \to
    \, \End(p_\alpha \Ind_H^\Gamma(\beta))^\Gamma
\end{equation*}
is such that we have natural isomorphisms
\begin{equation*}
    \ker(\pi_\alpha) \, \simeq \, \oplus_{j \notin J}
    \Ind_H^\Gamma(\End(\beta_j^{k_j}))^\Gamma \ \mbox{ and }
    \ \operatorname{Im}(\pi_\alpha) \, \simeq \, \oplus_{j \in J}
    \Ind_H^\Gamma(\End(\beta_j^{k_j}))^\Gamma \,.
\end{equation*}
\end{proposition}

\begin{proof}
By Lemma \ref{lemma.componentwise}, we can assume that $N = 1$. By
Lemma \ref{lemma.abelian.induction}, $\Ind_H^\Gamma(\beta_1^{k_1})
\simeq \oplus_{\chi \in \widehat{\Gamma}} V_\chi$, with $\chi\vert_{H}
= \beta_1$. We obtain
\begin{equation}
   p_\alpha \Ind_H^\Gamma(\beta_1^{k_1}) \, \simeq \
   \begin{cases}
     \ \alpha^{k_1} & \mbox{ if } \ \alpha\vert_{H} = \beta_1 \\
     \ \, 0 & \mbox{ otherwise.}
   \end{cases}
\end{equation}
Thus, if $\alpha \vert_{H} \neq \beta_1$, $J = \emptyset$, $\pi_\alpha =
0$, and hence
\begin{equation*}
  \begin{gathered}
  \ker(\pi_\alpha) \seq \Ind_H^\Gamma(\End(\beta_1^{k_1}))^\Gamma \seq
  \oplus_{j \notin J} \Ind_H^\Gamma(\End(\beta_j^{k_j}))^\Gamma
  \ \mbox{ and } \\
  \operatorname{Im}(\pi_\alpha) \seq 0 \seq \oplus_{j \in J}
  \Ind_H^\Gamma(\End(\beta_j^{k_j}))^\Gamma \,,
  \end{gathered}
\end{equation*}
as claimed.

Let us assume now that that $\alpha \vert_{H} = \beta_1$. The morphism
$\pi_\alpha$ can then be written as the composition
\begin{multline*}
  \pi_\alpha : \Ind_H^\Gamma(\End(\beta_1^{k_1}))^\Gamma \, \simeq \,
  \End(\beta_1^{k_1})^H \, \simeq \, M_{k_1}(\CC) \,
  \stackrel{\Psi_\alpha}{\longrightarrow} \, M_{k_1}(\CC) \\
  \simeq \, \End(\alpha^{k_1})^\Gamma \, \simeq \, \End(p_\alpha
  \Ind_H^\Gamma(\beta_1^{k_1}))^\Gamma\,.
\end{multline*}
To prove our proposition in this case ($\chi\vert_{H} = \beta_1$, and
hence in general, by Lemma \ref{lemma.componentwise}), it thus
suffices to show that $\Psi_\alpha = id$, which equivalent to the fact that
$\Psi_\alpha \neq 0$ for $\alpha \vert_{H} = \beta_1$. 

To prove this (that $\Psi_\alpha \neq 0$ for $\alpha \vert_{H} =
\beta_1$), let us begin by noticing that the morphism
$\Ind_H^\Gamma(\End(\beta)) \to \End(\Ind_H^\Gamma(\beta))$ is
injective. Hence $\Ind_H^\Gamma(\End(\beta))^\Gamma
\to \End(\Ind_H^\Gamma(\beta))^\Gamma$ is injective as well. This
means that not all of the maps $\Psi_\rho$, with $\rho \vert_{H} =
\beta_1$, can be zero, by Lemma \ref{lemma.abelian.induction}.  On the
other hand, it can be checked by a direct calculation that the action
of $\widehat{\Gamma/H}$ on $\Ind_H^\Gamma(\beta)$ of the same lemma
permutes the morphisms $\Psi_\rho$. It follows either that they are
all zero or that they are all non-zero. As their direct sum is
non-zero, it follows that all $\Psi_\rho \neq 0$, $\rho \vert_{H} =
\beta_1$.
\end{proof}

\color{black}

\subsection{The primitive ideal spectrum}\label{ssec.primitive.spec}
Let us begin by recalling a few basic facts about $C^{*}$-algebras
\cite{Dixmier}.  Recall that a two-sided ideal $I \subset A$ of a
$C^{*}$-algebra $A$ is called {\em primitive} if it is the kernel of
an irreducible, non-zero $*$-representation of $A$ (so $A$ is {\em
  not} considered to be a primitive ideal of itself). We will denote
by $\Prim(A)$ the {\em primitive ideal spectrum} of $A$, which, we
recall, is defined as the set of primitive ideals of $A$. For
instance, if $A = \maC_0(X)$, the space of continuous functions $X \to
\CC$ vanishing at infinity on the locally compact space, then we have
a natural identification (homeomorphism)
\begin{equation}\label{eq.Gelfand}
    \Prim(\maC_0(X)) \, \simeq \, X\,,    
\end{equation}
which is an identification that lies at the heart of non-commutative
geometry \cite{ConnesNCG, ConnesBook}. All the $C^*$-algebras
considered in this paper are type I. The relevance of this fact, for
us, is that, if $A$ is type I, then $\Prim(A)$ identifies with the set
of isomorphism classes of irreducible representations of $A$. Any
$C^*$-algebra with only finite dimensional representations is a type I
algebra \cite{Dixmier}. All the algebras considered in this paper
(except the algebras of compact operators on various Hilbert spaces),
have this property.

\begin{remark}\label{rem.group}
The following example will be used several times. Let $H$ be a finite
group and $\beta = \oplus_{j=1}^N \beta_j^{k_j}$ be a finite
dimensional $H$-module with $\beta_j$ non-isomorphic simple
$H$-modules. Since $\Hom_{H}(\beta_j^{k_j}, \beta_j^{k_j}) \simeq
M_{k_j}(\CC)$ and $ \Hom(\beta_i^{k_i}, \beta_j^{k_j})^{H} =
0$ for $i \neq j$ by the assumption that the simple $H$-modules
$\beta_i$ and $\beta_j$ are non-isomorphic, we obtain
\begin{multline*}
  \maL(\beta)^H \seq \End_H(\beta) \, \simeq \, \Hom_{H}(\oplus_i
  \beta_i^{k_i}, \oplus_j \beta_j^{k_j}) \, \simeq \, \oplus_{i,j}
  \Hom_{H}(\beta_i^{k_i}, \beta_j^{k_j}) \\
  \simeq \, \oplus_{j} \End_{H}(\beta_j^{k_j}) \, \simeq \, \oplus_{j}
M_{k_j}(\CC)\,.
\end{multline*}
The algebra $\maL(\beta)^H = \End_H(\beta)$ is thus a $C^*$-algebra
with only finite dimensional representations and we have natural
bijections
\begin{equation*}
  \Prim(\End_H(\beta)) \, \leftrightarrow \, \{\beta_1, \beta_2,
  \ldots, \beta_N\} \, \leftrightarrow \, \{1, 2, \ldots, N\}\,.
\end{equation*}
The algebra $\maL(\beta)^H$ is thus a semi-simple complex algebra with
simple factors $\End_{H}(\beta_j^{k_j}) \simeq M_{k_j}(\CC)$, $j = 1,
2, \ldots, N$.
\end{remark}

We shall need the connection between ideals and the topology of the
primitive ideal spectrum, which we now recall. Let $J$ be a closed,
two-sided ideal of $A$. Then
\begin{equation}\label{eq.ideals}
  \begin{gathered}
  \Prim(J) \seq \{ I \in \Prim(A)\vert\, J \not\subset I\}\ \mbox{
    and}\\
  \Prim(A/J) \, \cong \, \{ I \in \Prim(A)\vert\, J \subset I\} \seq
  \Prim(J)^c \,.
 \end{gathered}
\end{equation}
The second bijection sends an ideal $I \supset J$ to
  $\pi_{A,J}^{-1}(I)$, where $\pi_{A,J} : A \to A/J$ is the canonical
  bijection. In the following, we shall identify the two sets using
  this bijection.

In analogy with the Zariski topology, $\Prim(A)$ is endowed with the
Jacobson (or hull-kernel) topology, which is the topology whose open
sets are those of the form $\Prim(J)$. Conversely, given an open
subset $U \subset \Prim(A)$, then $U = \Prim(J_U)$, where $J_U :=
\cap_{\mathfrak P \in \Prim(A) \smallsetminus U} {\mathfrak P}$. For
us, it will be convenient to regard these properties as a {\em
  one-to-one} correspondence between the closed, two-sided ideals of
$A$ and the closed subsets $\Prim(A/J)= \Prim(A) \smallsetminus
\Prim(J)$ of $\Prim(A)$. If $I \subset A$ is an ideal such that $a \in
A$ and $a I = 0$ implies $a = 0$, then $I$ is called an {\em essential
  ideal} of $A$ and its associated open set $\Prim(J)$ is dense in
$\Prim(A)$.

Let $Z$ be a commutative $C^*$-algebra and $\phi : Z \to M(A)$ be a
*-morphism to the multiplier algebra of $A$ \cite{APT, Busby}. Assume
that $\phi(Z)$ commutes with $A$ and $\phi(Z)A = A$. Then Schur's
lemma gives that there exists a natural continuous map $\phi^* :
\Prim(A) \to \Prim(Z)$, which we shall call also the ``central
character map'' (associated to $\phi$).

\subsection{Lie group actions on manifolds}
{\em We will assume from now on that $\Gamma$ is a compact Lie group,
  that $M$ is a Riemannian manifold, and that $\Gamma$ acts smoothly
  and isometrically on $M$.}

\subsubsection{Slices and tubes}
\label{ssec.slice}
Let us fix $x \in M$ arbitrarily. Then its orbit $\Gamma x$ is a
smooth, compact submanifold of $M$. Let $N_x$ be the
  orthogonal of $T_x(\Gamma x)$ in $T_xM$. (If $\Gamma$ is discrete,
  as in our main results, then, of course, $N_x = T_x M$.) Then the
isotropy group $\Gamma_x$ acts linearly and isometrically on
$N_x$. For $r > 0$, let $U_x := (N_x)_r$ denote the set of vectors of
length $< r$ in $N_x$. It is known then that, for $r > 0$ small
enough, the exponential map gives a $\Gamma$-equivariant isometric
diffeomorphism
\begin{equation}\label{eq.def.tube}
  W_x \simeq \Gamma \times_{\Gamma_x} U_x \ede \{ (\gamma, y) \in
  \Gamma \times U_x\, \vert \ (\gamma h, y) \equiv (\gamma , h y), \,
  h \in \Gamma_x \}\,,
\end{equation}
where $W_x$ is a $\Gamma$-invariant neighborhood of $x$ in $M$. More
precisely, $W_x$ is the set of $y \in M$ at distance $<r$ to the orbit
$\Gamma x$, if $r > 0$ is small enough. The set $W_x$ is called a {\em
  tube} around $x$ (or $\Gamma x$) and the set $U_x$ is called the
{\em slice} at $x$. The range of $r$ depends, of course, on $x$,
namely it must satisfy $0 < r < r_x$, for some $r_x > 0$. We assume
that, for each $x \in M$, such an $r \in (0, r_x)$ has been chosen and
we will keep fixed the notation for the slice $U_x$ and for the tube
$W_x$ associated to $x$ and to the fixed choice of $r$.

\subsubsection{Equivariant vector bundles}
Let $E \to M$ be a $\Gamma$-equivariant smooth vector bundle. All our
vector bundles will be assumed to be finite dimensional. Let us fix $x
\in M$ and consider the tube $W_x \simeq \Gamma \times_{\Gamma_x} U_x$
around $x$ of Equation \eqref{eq.def.tube}. We use this diffeomorphism
to identify $U_x$ with a subset of $M$, in which case, we can also
asume the restriction of $E$ to the slice $U_x$ to be trivial. More
precisely,
\begin{equation}\label{eq.trivial.E}
  \begin{gathered}
    E\vert_{U_x} \, \simeq \, U_x \times \beta \ \mbox{ and }\\
    E\vert_{W_x} \, \simeq \, \Gamma \times_{\Gamma_x} (U_x
   \times \beta)\,,
  \end{gathered}
\end{equation}
for some $\Gamma_x$-module $\beta$, the second isomorphism being
$\Gamma$-equivariant.

Let $F \to Y$ be a Hermitian vector bundle on a locally compact
measure space. Then $L^2(Y; F)$ denotes the set of (equivalence
classes) of square integrable sections of the vector bundle $F$ and
$\maC_0(Y; F)$ denote the set of its continuous sections vanishing at
infinity.

\subsubsection{The principal orbit bundle}
\label{ssec.principal}
Let us assume from now on that $M$ is connected, except
  where explicitly stated otherwise. The reason of this assumption is that it is known
then \cite{tomDieckTransBook} that there exists a {\em minimal
  isotropy} subgroup $\Gamma_0 \subset \Gamma$, in the sense that
$M_{(\Gamma_0)}$ is a dense open subset of $M$. (Recall that $M_{(H)}$
denotes the set of points of $M$ whose stabilizer is conjugated in
$\Gamma$ to $H$.) In particular, there exist minimal elements for
inclusion for the set of isotropy groups of points in $M$ and all
minimal isotropy groups are conjugated (to a fixed subgroup, denoted
$\Gamma_0$ in what follows). The notation $\Gamma_0$ will remain fixed
from now on. By the definition, the set $M_{(\Gamma_0)}$ consists of
the points whose stabilizer is conjugated to that minimal
subgroup. The set $M_{(\Gamma_0)}$ is called the {\em principal orbit
  bundle} of $M$.

If $x \in M_{(\Gamma_0)}$, then $\Gamma_x$ acts {\em trivially} on the
slice $U_x$ at $x$, by the minimality of $\Gamma_0$. (See
\ref{ssec.slice} for notation.) In general, for $x \in M$, the
isotropy of $\Gamma_x$ acting on $U_x$ will contain a subgroup
conjugated to $\Gamma_0$.

If $\Gamma$ is abelian, there is only one minimal isotropy group
  $\Gamma_0$ (recall that we are assuming $M$ to be
  connected). Moreover, we can then factor the action of $\Gamma$ to
  an action of $\Gamma/\Gamma_0$ on $M$, which has trivial minimal
  isotropy, that is, it is free on a dense, open subset of $M$.

\subsection{Pseudodifferential operators}
We continue to assume that $\Gamma$ is a compact Lie group that acts
smoothly and isometrically on a smooth Riemannian manifold $M$. We let
$\op{m}$ denote the space of order $m$, {\em classical}
pseudodifferential operators on $M$ with {\em compactly supported}
distribution kernel. Recall that, in this article, we consider only
classical pseudodifferential operators. We let $\clop$ and $\clopn$
denote the norm closures of $\op{0}$ and $\op{-1}$, respectively. The
action of $\Gamma$ then extends to an action on $E$ and on $\op{m}$,
$\clop$, and $\clopn$. We will denote by $\maK(\maH)$ the algebra of
compact operators acting on a Hilbert space $\maH$. We will write
$\maK$ instead of $\maK(\maH)$ when the Hilbert spaces $\maH$ is clear
from the context. We have
\begin{equation}\label{eq.compact}
  \clopn \seq \maK(L^2(M; E)) \,,
\end{equation}
since we have considered only pseudodifferential operators with
compactly supported distribution kernels.

Let $S^*M$ denote the {\em unit cosphere bundle} of 
$M$, that is, the set of unit vectors in $T^*M$, as
usual. We will denote, as usual, by $\maC_0(S^*M; \End(E))$
the set of continuous sections of the {\em lift} of the vector bundle
$\End(E) \to M$ to $S^*M$. 

\begin{corollary}\label{cor.full.symbol}
We have an exact sequence
\begin{equation*}
   0 \, \to \, \maK^\Gamma = \clopn^\Gamma \, \to \,
   \clop^\Gamma\, \stackrel{\sigma_{0}}{-\!\!\!\longrightarrow}\,
   \maC_0(S^*M; \End(E))^\Gamma \, \to \, 0 \,.
\end{equation*}
\end{corollary}

\begin{proof}
Recall then
that the principal symbol extends to a continuous, surjective,
$\Gamma$-equivariant map
\begin{equation}
  \sigma_0 \seq \sigma_0^{M} : \clop\, \to\, \maC_0(S^*M; \End(E))
\end{equation} 
with kernel $\clopn = \maK = \maK(L^2(M; E))$. In other words, we have
the following well known exact sequence
\begin{equation*}
   0 \, \to \, \maK \, \to \, \clop\,
   \stackrel{\sigma_{0}}{-\!\!\!\longrightarrow}\,
   \maC_0(S^*M; \End(E)) \, \to \, 0.
\end{equation*}
The exact sequence of the corollary is obtained from the fact that the
functor $\maH \to \maH^{\Gamma}$ is exact on the category of
$\Gamma$--modules with continuous $\Gamma$-action, since $\Gamma$ is
compact.
\end{proof}

\subsection{Reduction to order-zero operators}
We assume here that $M$ is a compact Riemannian manifold with metric
$g$. Let us fix a $\Gamma$-invariant metric on $E \to M$. Let $\nabla$
be a metric preserving, $\Gamma$-invariant connection on $E$. In what
follows $\Delta := \Delta_g^E = \nabla^* \nabla$ will denote the
(positive) Laplacian on $M$ with coefficients in $E$. Since $M$ is
compact, the Sobolev space $H^s(M; E)$ can be defined for any $s \in
\RR$ as the domain of $(1+\Delta)^{s/2}$ and $(1+\Delta)^{s/2} :
H^s(M; E) \to L^2(M; E)$ is an isomorphism.

\begin{lemma}\label{lemma.red.zero}
Assume that $M$ is compact. We have that a bounded operator $P
: H^{s}(M; E) \to H^{s-m}(M; E)$ is Fredholm if, and only if,
$\widetilde P := (1 + \Delta)^{(s-m)/2} P (1 + \Delta)^{-s/2} : L^2(M;
E) \to L^2(M; E)$ is Fredholm. Moreover, if $P$ is the limit in
$\maL(H^{s}(M; E), H^{s-m}(M; E))$ of a sequence of operators $P_n \in
\op{m}$, then $\widetilde P \in \clop$.
\end{lemma}

\begin{proof}
The first part follows from the fact $(1 + \Delta)^{s} : H^{m}(M; E)
\to H^{m-2s}(M; E)$ is an isomorphism for all $m , s \in \RR$, since
$M$ is compact. The second part follows from the well known fact that the powers
of the Laplace operator are classical pseudodifferential operators
\cite{Seeley67} and the
continuity of the multiplication in the operator norm.
\end{proof}

The case of operators in $\psi^m(M; E, F)$ acting between {\em two}
vector bundles $E, F \to M$ can be reduced to the case of a single
vector bundle by considering $\psi^m(M; E \oplus F)$. Moreover, $P
\in \overline{\psi}^0(M; E, F) \subset \overline{\psi}^{0}(M; E
\oplus F)$ will be Fredholm if, and only if,
\begin{equation*}
\left [
\begin{array}{cc}
0 & P \\
P^* & 0 
\end{array}
\right ] \ \in \ \overline{\psi}^{0}(M; E \oplus F)
\end{equation*}
is Fredholm (here all operators act on $L^2$-spaces). Therefore it is
sufficient to consider only the case of order-zero operators acting on
a single vector bundle.

\section{The structure of regularizing operators}
\label{sec.structure.of.regularizing.operators}
We continue to assume that $M$ is a complete Riemannian manifold and
that $\Gamma$ is a compact Lie group acting by isometries on $M$. From
now on, all our vector bundles will be $\Gamma$-equivariant vector
bundles.

As explained in the Introduction, we want to identify the structure of
the restrictions of $\Gamma$-invariant pseudodifferential operators on
$M$ to the isotypical components of $L^2(M; E)$. Let $\pi_\alpha$ be
this restriction morphism to the $\alpha$-isotypical component. More
precisely, we want to understand the structure of the algebra
$\pi_\alpha(\clop^\Gamma)$, for any fixed $\alpha \in \Gamma$. See
Equations \eqref{eq.restriction} and \eqref{eq.restriction2} for the
definition of the restriction morphism $\pi_\alpha$ and of the
projectors $p_\alpha \in C^*(\Gamma)$.

In this section, we study two basic cases: that of inner actions and
that of free actions (of $\Gamma$).

\subsection{Inner actions of $\Gamma$: the abstract case}
In this subsection we deal with the case when the action of $\Gamma$
is implemented by unitaries in the multiplier algebra (the case of
inner actions). This allows us, in particular, to settle the case of
regularizing operators. We shall need the following notion of direct
sum. Recall that $M(A)$ denotes the multiplier algebra of a
$C^*$-algebra $A$. Recall the following standard definition.

\begin{definition} \label{def.inner}
Let $\phi : \Gamma \to \Aut(A)$ be the action of a group $\Gamma$ by
automorphisms on a $C^*$-algebra $A$. We shall say that this action is
{\em inner} if the morphism $\phi$ lifts to a morphism $\psi : \Gamma
\to U(M(A))$ to the group of unitary elements of $M(A)$ such that
\begin{equation*}
  \phi_g(a) \seq \psi(g) a \psi(g)^{-1} \,, \quad a \in A, \ g \in
  \Gamma\,.
\end{equation*}
\end{definition}

\begin{remark}\label{rem.inner}
If $\alpha : \Gamma \to \Aut(A)$ is an inner action as in Definition
\ref{def.inner}, then we obtain, in particular, a morphism
$C^*(\Gamma) \to M(A)$. If, moreover, $\pi : A \to \maL(\maH)$ is a
$*$-representation, then it extends to a representation of
$M(A)$. This induces naturally a unitary representation of $\Gamma$ on
$\maH$. This representation is uniquely determined if $\pi$ is
non-degenerate (i.e. if $\pi(A)\maH$ is dense in $\maH$).
\end{remark}

If $A_n$, $n \ge 1$ is a sequence of $C^*$-algebras, we shall denote
by $c_0 -\oplus_{n=1}^\infty A_n$ the inductive limit $\lim_{N \to
  \infty} \oplus_{n=1}^N A_n$. This definition extends immediately to
countable families of $C^*$-algebras.  Recall that $\widehat{\Gamma}$
denotes the set of isomorphism classes of irreducible unitary
representations of $\Gamma$. We shall need then the following general
result.

\begin{proposition}\label{prop.action}
Let $A$ be a $C^*$-algebra with a *-morphism $C^*(\Gamma) \to M(A)$
and let $p_\alpha \in C^*(\Gamma)$ denote the central projector
corresponding to $\alpha \in \widehat{\Gamma}$. Let $A^\Gamma := \{a
\in A \, \vert \ a f = f a, \ f \in C^*(\Gamma) \}$.  Then
\begin{equation}
   A^\Gamma \, \simeq\, c_0-\oplus_{\alpha \in \hat{\Gamma}} \
   p_\alpha A^\Gamma \,.
\end{equation}
If $I \subset A$ is a closed two-sided ideal, then $C^*(\Gamma)I
\subset I$ and hence we obtain *-morphisms $C^*(\Gamma) \to M(I)$ and
$C^*(\Gamma) \to M(A/I)$ such that the induced sequence
\begin{equation*}
  0 \to p_\alpha I^\Gamma \to p_\alpha A^\Gamma \to p_\alpha
  (A/I)^\Gamma \to 0
\end{equation*}
is exact for any $\alpha \in \widehat{\Gamma}$. 
\end{proposition}

\begin{proof}
Let us arrange the elements of $\widehat{\Gamma}$ in a sequence
$\rho_n$, $n \ge 1$. Then, for any $a \in A$, $\lim_{N \to \infty}
\sum_{n=1}^N p_{\rho_n} a = a$. Since, for any $a \in A^{\Gamma}$, we
have $p_\alpha a = a p_\alpha$, the isomorphism $ A^\Gamma \simeq
c_0-\oplus_{\alpha \in \hat{\Gamma}}\, A^\Gamma p_\alpha$ follows.
Finally, it is known that if $I$ is a two-sided ideal of $A$, then
$M(A)I \subset I$.
\end{proof}

If $A \subset \maL(\maH)$ is a sub-$C^*$-algebra of the algebra of
bounded operators on a Hilbert space $\maH$ together with a compatible
$\Gamma$-module structure on $\maH$, we let $\pi_\alpha : A^\Gamma \to
\maL(\maH_\alpha)$ be the restriction morphism to the
$\alpha$-isotypical component, as before, see Equations
\eqref{eq.restriction} and \eqref{eq.restriction2}.

\begin{proposition}\label{prop.action2}
In addition to the assumptions of Proposition \ref{prop.action}, let us suppose
that $A \subset \maL(\maH)$, for some Hilbert space $\maH$. Let $\alpha \in
\widehat{\Gamma}$, as before. Then the morphism $\pi_\alpha : A^\Gamma \to
\maL(\maH_\alpha)$ restricts to an isomorphism $p_\alpha A^\Gamma \to
\pi_\alpha(A^\Gamma)$.
\end{proposition}

\begin{proof} 
Let us recall that, as explained in Remark \ref{rem.inner}, the
  group $\Gamma$ will be represented on $\maH$.
  The rest follows from Proposition \ref{prop.action}, whose notation
we shall use freely. Indeed, if $\alpha \neq \beta \in \widehat
\Gamma$, then $\pi_\alpha (p_\beta) = 0$. On the other hand
$\pi_\alpha(p_\alpha) = p_\alpha$. The result follows by combining
this property with $\maH_\alpha = p_\alpha \maH$ and with Proposition
\ref{prop.action}.
\end{proof}

\subsection{Inner actions of $\Gamma$: pseudodifferential operators}
We now apply the results of the previous subsection to the
  algebra of pseudodifferential operators. In particular, this gives a
  rather complete picture for the case of negative order operators
(recall that the closure of regularizing operators coincides with that
of negative order operators). Since we are eventually interested only
in the case $\Gamma$ finite, we discuss only briefly the issues
related to the non-discrete case (such as the continuity of the action
of $\Gamma$).

A crucial first observation is that if $\gamma \in \Gamma$ and $P \in
\op {-\infty}$, then $\gamma P, P \gamma \in \op
    {-\infty}$. This leads to several interesting consequences. We
    record this as a lemma.

\begin{lemma} \label{lemma.multiplier}
We have $\gamma \op {-\infty} = \op {-\infty} \gamma = \op {-\infty}$,
for all $\gamma \in \Gamma$, and the induced actions (to the right and
to the left) of $\Gamma$ on $\op {-\infty}$ are continuous and unitary
in the operator norm. Consequently, we have
\begin{equation*}
   C^*(\Gamma) \clopn \, + \, \clopn C^*(\Gamma) \ \subset \ \clopn \,
   .
\end{equation*}
\end{lemma}

\begin{proof}
Since smoothing operators have smooth kernels $k(x,y) \in
\mathrm{Hom}(E_y,E_x)$, the action of $\Gamma$ induced an action on
$\op {-\infty}$. Thanks to the $\Gamma$-invariance of the metric on
$M$, we have $\|\gamma P\| = \|P\gamma\| = \|P\|$ in the
$L^2$-operator norm, for any $P \in \op{-\infty}$. The second part
follows from the first part.
\end{proof}

We then have the following:

\begin{proposition} \label{prop.multiplier}
The multiplication by $\Gamma$ on $\clopn$ defines a *-morphism
$C^*(\Gamma) \to U(M(\clopn))$ to the multiplier algebra of
$\clopn$.
\end{proposition}

\begin{proof}
  This follows from Lemma \ref{lemma.multiplier}.
\end{proof}

We obtain the following corollary.

\begin{corollary}  \label{cor.exact}
Let $A = \clopn$. If $\Gamma$ acts trivially on $M$ (so that the
action of $C^*(\Gamma)$ extends to the algebra $\clop$), we also allow
$A = \clop$ or $A = \maC_0(S^*M; \End(E))$. We then have isomorphisms
\begin{equation*}
	A^\Gamma \ \simeq \ c_0-\oplus_{\alpha \in
          \widehat{\Gamma}}\ p_\alpha A^\Gamma\, .
\end{equation*}
Moreover, if $\Gamma$ acts trivially on $M$ and $\alpha \in \hat
\Gamma$, we have an exact sequence
\begin{equation*}
   0 \to p_\alpha \clopn^\Gamma \to p_\alpha \clop^\Gamma \to
   \maC_0(S^*M; \End(p_\alpha E)^\Gamma) \to 0 \, .
\end{equation*}
\end{corollary}

\begin{proof}
The first part follows from Proposition \ref{prop.action} applied to
the algebra $\clop$ and to its ideal $\maK$. The second part
follows from the exactness of the functors $V \to V^\Gamma$ and $V \to
p_\alpha V$ on the category of $\Gamma$-modules.
\end{proof}

Let $\alpha \in \hat \Gamma$ and let $\pi_\alpha$ be the
representation of $\clop^\Gamma$ on $L^2(M; E)_{\alpha}$ defined by
restriction as before, Equations \eqref{eq.restriction} and
\eqref{eq.restriction2}.  The assumptions of Proposition
  \ref{prop.action2} are satisfied for $A = \clopn$, so we obtain the
  following.

\begin{corollary} \label{cor.isomorphism} The morphism $\pi_\alpha$
restricts to an isomorphism from $p_\alpha \clopn^\Gamma$ to
$\pi_\alpha(\clopn^\Gamma)$.
\end{corollary}

We also have the following simple result, which makes the last
corollary more precise. Recall that $\maK \simeq \clopn$. This allows
us to better describe the structure of $\clopn^\Gamma$.

\begin{proposition} \label{prop.image} The algebra
$\pi_\alpha(\maK^\Gamma)$ is the algebra of $\Gamma$--equivariant
  compact operators on $L^2(M; E)_\alpha$.
\end{proposition}

\begin{proof}
Let $T \in \maK$ commute with $\Gamma$.  Then its restriction to a
$\Gamma$-invariant subspace is still compact and still commutes with
$\Gamma$. This shows that $\pi_\alpha(\maK^\Gamma)$ is contained in
the set $K_\alpha$ of $\Gamma$-invariant compact operators on
$L^2(M; E)_\alpha$. Conversely, $K_\alpha \subset \maK^\Gamma$ and
$\pi_\alpha$ acts as the identity on $K_\alpha$.
\end{proof}

\subsection{The case of free actions}
Let us now tackle the opposite case, that is when $\Gamma$ acts freely on
$M$. We shall assume that $\Gamma$ is finite, for simplicity (we only
need this case), and hence, in particular, the action of $\Gamma$ is
proper. We have then the following well-known result (see
\cite{ConnesBook, documenta} and the references therein).

\begin{proposition}\label{prop.free}
Let us assume that $\Gamma$ is a finite group acting freely on $M$ and
let $\alpha(\gamma) = 1$ be the trivial representation, so $\pi_\alpha
= \pi_1$. Let us denote by $F:=E/\Gamma \to M/\Gamma$ the resulting
vector bundle and $\phi : \maC_0(S^*M; \End(E))^{\Gamma} \to
\maC_0(S^*M/\Gamma; \End( E/\Gamma))$ the resulting isomorphism. Then
we have the following morphism of exact sequences, with the vertical
arrows surjective.
\begin{equation*}
\xymatrix{0 \ar[r]& \clopn^\Gamma \ar[r]\ar[d]_{\pi_1} & \clop^\Gamma
  \ar[r] \ar[d]^{\pi_1} & \maC_0(S^*M; \End( E))^\Gamma \ar[r]
  \ar[d]^{\phi} &0 \\
  0 \ar[r]& \overline{\psi^{-1}}(M/\Gamma, F) \ar[r]&
  \overline{\psi^{-0}}(M/\Gamma, F) \ar[r] &
  \maC_0(S^*M/\Gamma; \End(F)) \ar[r]&0 }
\end{equation*}
\end{proposition}

For $G \subset M$, let $A_G \ede \maC_0(S^*G; \End(E))$ and consider
the {\em surjective} map
\begin{multline}\label{eq.def.maRG}
  \maR_G \, : \, A_G^\Gamma \ede \maC_0(S^*G; \End(E))^\Gamma \ \simeq
  \ \overline{\psi^{0}}(G; E)^\Gamma / \overline{\psi^{-1}}(G;
  E)^\Gamma \\
  \to \pi_\alpha( \overline{\psi^{0}}(G;
  E)^\Gamma)/\pi_\alpha(\overline{\psi^{-1}}(G; E)^\Gamma) \,.
\end{multline}

\begin{proposition}\label{prop.free.isomorphism}
Let $\Gamma$ be a finite group acting on a smooth
compact manifold $M$ (without boundary). Assume that the action of
$\Gamma$ is free on a dense, open subset of $M$. Let $E \to M$ be an
equivariant vector bundle. Then the map $\maR_M :
\maC_0(S^*M; \End(E))^\Gamma \to \pi_1( \overline{\psi^{0}}(M;
E)^\Gamma)/\pi_1(\overline{\psi^{-1}}(G; E)^\Gamma)$ of Equation
\eqref{eq.def.maRG} is injective, and hence an isomorphism of
algebras.
\end{proposition}

\begin{proof}
Let $M_0 \subset M$ be an open, dense subset on which $\Gamma$ acts
freely. Proposition \ref{prop.free} for $M$ replaced with $M_{0}$
shows that $\maR_{M_0}$ is injective. Hence $\maR_{M}$ is injective on
$\maC_0(S^*M_{0}; \End(E))^\Gamma$, because the restriction of
$\maR_{M}$ to $\maC_0(S^*M_{0}; \End(E))^\Gamma$ is
$\maR_{M_0}$. Since the later is an essential ideal in
$\maC_0(S^*M; \End(E))^\Gamma$, it follows that $\maR_M$ is also
injective (everywhere on $\maC_0(S^*M; \End(E))^\Gamma$).
\end{proof}

\section{The principal symbol}
\label{sec.principal.symbol}

Let us fix an irreducible representation $\alpha$ of $\Gamma$ and
consider the fundamental restriction morphism $\pi_\alpha$ of Equation
\eqref{eq.restriction}. See also Subsection
\ref{ssec.representations}, especially Equation
\eqref{eq.restriction2}, for more details on the morphism
$\pi_\alpha$. We are mostly concerned with the morphism $\pi_\alpha :
\clop^\Gamma \to \maL(L^2(M; E)_\alpha)$ and, in this section, we
identify the quotient
\begin{equation*}
  \pi_\alpha(\clop^\Gamma)/\pi_\alpha(\clopn^\Gamma) \,.
\end{equation*}
Since $\pi_\alpha(\clopn^\Gamma)$ was identified in the previous
section, information on the above quotient algebra will give further
insight into the structure of the algebra $ \pi_\alpha(\clop^\Gamma)$
and will provide us, eventually, with Fredholm conditions.

In the beginning of this section, we continue to assume that $M$ is a
complete Riemannian manifold and that $\Gamma$ is a compact Lie group
acting by isometries on $M$.  Since the results are different in the
discrete and non-discrete case, we will assume beginning with
Subsection \ref{ssec.isotropy} that $\Gamma$ is finite. Moreover, for
the main result, we shall assume that $\Gamma$ is abelian, since the
abelian case is simpler and presents some additional features. Some
intermediate results are true only in the abelian case. It is also the
abelian case that will be used for our application to boundary value
problems.

\subsection{The primitive ideal spectrum of the symbol algebra}
We now turn to the description of the primitive ideal spectrum of the
algebra $\maC_0(S^*M; \End(E))^\Gamma$ of symbols. For simplicity, for $O
\subset M$ open, we denote again
$A_O := \maC_0(S^*O; \End(E))$, as in the definition of the morphism
$\maR_O$ of Equation \eqref{eq.def.maRG}.
We shall be mostly concerned with the cases $O = M$ and $O = O_0 :=
M_{(\Gamma_0)}$. We have the following standard result.

\begin{proposition}\label{prop.str2}
The algebra $Z_M := \maC_0(S^*M)^\Gamma = \maC_0(S^*M/\Gamma)$
identifies with a central subalgebra of $A_M^\Gamma :=
\maC_0(S^*M; \End(E))^\Gamma$. Let $z_\xi$ be the maximal ideal of
$Z_M$ associated to the orbit $\Gamma \xi$ for some $\xi \in S_x^*M$,
let $E_x$ be a fiber of $E$ corresponding to $\xi$, and let $E_x
\simeq \oplus_{j=1}^N \beta_j^{k_j}$ be its decomposition into
$\Gamma_\xi$-isotypical components, with $\beta_j$ simple,
non-isomorphic $\Gamma_\xi$ modules. Then $A_M^{\Gamma}/z_\xi
A_M^{\Gamma} \simeq \End(E_x)^{\Gamma_\xi}$ is a semi-simple,
finite-dimensional (complex) algebra with $N$ simple factors
$\End_{\Gamma_\xi}(\beta_j^{k_j}) \simeq M_{k_j}(\CC)$, $j \in \{1, 2,
\ldots, N\}$.
\end{proposition}

\begin{proof}
We have $\maC_0(M) \subset \maC_0(M; \End(E)) \subset \End(\maC_0(M;
E))$, with $f \in \maC_0(M)$ acting as a scalar on each fiber
$E_x$. In fact, this identifies $\maC_0(M)$ with the center
$Z(\maC_0(M; \End(E)))$ of $\maC_0(M; \End(E))$. By considering
$\Gamma$ invariant functions, we obtain that $Z_M :=
\maC_0(S^*M)^\Gamma = \maC_0(S^*M/\Gamma)$ is contained in the center
$Z(A_M^\Gamma)$ of $A_M^\Gamma := \maC_0(S^*M; \End(E))^\Gamma$. Let
$\Gamma\xi$ denote the orbit in $S^*M$ that we consider and let $J$ be
the (non-maximal, in general) ideal of $\maC_0(S^*M)$ corresponding to
functions vanishing on this orbit. Then $J$ is $\Gamma$ invariant and
$J^\Gamma = z_\xi$. By taking the $\Gamma$ invariants in the exact
sequence $0 \to JA_M \to A_M \to A_M/JA_M \to 0$ and using
  Frobenius reciprocity for $A_M/JA_M \simeq
  \Ind_{\Gamma_\xi}^\Gamma(\End(E_x))$, we obtain that
\begin{equation}
  A_M^\Gamma/z_\xi A_M^\Gamma \, \simeq \, (A_M/JA_M)^\Gamma \,
  \simeq \End(E_x)^{\Gamma_\xi} \,.
\end{equation}
The proof is completed using Remark \ref{rem.group} for $H =
\Gamma_{\xi}$.
\end{proof}

See \cite{BGreenR, EchterhoffWilliams14, GreenActa, Rosenberg81,
  williamsBook} for similar results. Recall that if $\phi : Z \to
M(A)$ is a central *-morphism (i.e. $\phi(z)a = a\phi(z)$, for $a\in
A$ and $z \in Z$) such that $\phi(Z)A = A$, then it defines a natural
``central character'' map $\phi^* : \Prim(A) \to \Prim(Z)$ by Schur's
Lemma. The same proof yields the following.

\begin{corollary}\label{cor.structure}
There is a one-to-one correspondence between the primitive ideals of
$A_M^\Gamma := \maC_0(S^*M; \End(E))^\Gamma$ and the $\Gamma$-orbits
of the pairs $(\xi, \rho)$, where $\xi \in S_x^*M$ and $\rho \in
\widehat \Gamma_\xi$ appears in $E_x$ (i.e. $\Hom_{\Gamma_\xi}(\rho,
E_x) \neq 0$). The group $\Gamma$ acts by joint conjugation on both
$\xi$ and $\rho$. The inclusion $Z_M := \maC_0(S^*M)^\Gamma \to
A_M^\Gamma$ is such that the associated canonical central character
map of spectra
\begin{equation*}
  \Prim(A_M^\Gamma) \ \to \ \Prim(Z_M) \seq S^*M/\Gamma
\end{equation*}
is continuous, finite-to-one, and maps the orbit $\Gamma(\xi, \rho)$
to the orbit $\Gamma \xi$.
\end{corollary}

\begin{proof}
Let $\pi$ be an irreducible representation of $A_M^\Gamma :=
\maC_0(S^*M; \End(E))^\Gamma$. Then $\pi$ is a multiple of a character
on $Z_M := \maC_0(S^*M)^\Gamma \subset Z(A_M^\Gamma)$, by Schur's
lemma. Let this character correspond to the orbit $\Gamma \xi \in
S^*M/\Gamma$, with $\xi \in S^*_xM$ and denote by $z_\xi$ the
corresponding maximal ideal of $Z_M$, as in the proof of the last
lemma. In other words, $z_\xi$ is the value of the central character
map corresponding to the inclusion $Z_M \subset A_M^\Gamma$ applied to
$\pi$. Then $\pi$ factors out through an irreducible representation
of $A_M^{\Gamma}/z_\xi A_M^{\Gamma} \simeq \End(E_x)^{\Gamma_\xi}$.
Let us write $E_x \simeq \oplus_{j=1}^N \beta_j^{k_j}$ with $\beta_j$
non-isomorphic simple $\Gamma_\xi$ modules, as in the statement of
Proposition \ref{prop.str2}. Then $\End(E_x)^{\Gamma_\xi} \simeq
\oplus_{j=1}^N \End_{\Gamma_\xi}(\beta_j^{k_j})$, a direct sum of
simple algebras. Thus $\pi$ factors through one of the simple algebras
$\End_{\Gamma_\xi}(\beta_j^{k_j})$ This associates to $\pi$ the pair
$(\xi, \rho) = (\xi, \beta_j)$, as desired. This pair is not unique,
but depends on the choice of $\xi$. It becomes unique modulo the
action of $\Gamma$, however. Conversely, given such a pair $(\xi,
\rho)$, we obtain an irreducible representation of $A_M^\Gamma$
following exactly the same procedure in reverse order. The first part
of the result follows.

To prove that $\phi^*$ is finite to one, we notice that, by
  construction, $\phi^*(\xi, \rho) = \xi$. Since only a finite number
  of (isomorphism classes of) simple $\Gamma_\xi$ modules appears in
  $E_x$, the finiteness follows.
\end{proof}

\begin{remark}\label{rem.description}
Let us denote by $X_{M, E, \Gamma}$ the set of pairs $(\xi, \rho)$,
where $\xi \in T_x^*M\setminus\{0\}$ and $\rho \in \widehat
\Gamma_\xi$ appears in $E_x$ (i.e. $\Hom_{\Gamma_\xi}(\rho, E_x) \neq
0$), as in the statement of Corollary \ref{cor.structure}. The main
result of that corollary is a natural bijection
\begin{equation}\label{eq.bijection.Omega}
   X_{M, E, \Gamma}/\Gamma \, \simeq\, \Prim(A_M^\Gamma)\,.
\end{equation}
This bijection can be explicitely described as follows: to an orbit
$\Gamma(\xi,\rho)$ in $X_{M,E,\Gamma}$ is associated $\ker
(\pi_{\xi,\rho})$ in $\Prim A_M^\Gamma$, where for any $f \in
A_M^\Gamma$ we define $\pi_{\xi,\rho}(f)$ as the restriction of
$f(\xi)$ to the $\rho$-isotypical component of $E_x$.
\end{remark}

\subsection{Factoring out the minimal isotropy}
\label{ssec.isotropy}
Recall that we are assuming $M$ to be connected; in that case there is
a minimal isotropy type for the action of $\Gamma$. We shall also
assume from now on that $\Gamma$ is abelian, for the reasons discussed
in the Introduction. In particular, it is the case needed for our
applications to boundary value problems and, moreover, some results
are not true in the non-abelian case.

Let $\alpha \in \widehat{\Gamma}$ as before, and recall that we want
to determine the structure of the quotient
$\pi_\alpha(\clop^\Gamma)/\pi_\alpha(\clopn^\Gamma)$ of the restricted
algebras to the $\alpha$-isotypical component. To this end, recall the
morphism 
\begin{equation*}
  \maR_M : A_M^\Gamma \ede \maC_0(S^*M; \End(E))^\Gamma \, \to
\pi_\alpha( \overline{\psi^{0}}(M;
E)^\Gamma)/\pi_\alpha(\overline{\psi^{-1}}(M; E)^\Gamma)
\end{equation*} 
of Equation \eqref{eq.def.maRG}. The main result of this subsection is to
determine the kernel of this morphism.

The main reason why the abelian case is simpler than the general case
is that in the abelian case all minimal isotropy subgroups of $\Gamma$
acting on $M$ coincide. The (unique) minimal isotropy subgroup of
$\Gamma$ acting on $M$ will be denoted by $\Gamma_0$, as
before. Recall that the set $O_0 := M_{(\Gamma_0)}$ of points $x \in
M$ with isotropy $\Gamma_x = \Gamma_0$ is called the principal orbit
bundle of $M$; it is a dense, open subset of $M$. For every (other) $x
\in M$, we have $\Gamma_0 \subset \Gamma_x$.

We obtain that the group $\Gamma_0$ acts trivially on $M$. Moreover,
there exists a unitary group morphism (representation) $\Gamma_0
\to \End(E)$ that implements the action of $\Gamma_0$ on $\clop$. Let
$p^{(0)}_\beta \in C^*(\Gamma_0)$, $\beta \in \widehat{\Gamma}_0$, be
the central projectors associated to the irreducible representations
of $\Gamma_0$ (the additional exponent is to differentiate them from
the projectors $p_\alpha$, $\alpha \in \widehat{\Gamma}$). Corollary
\ref{cor.exact} then gives the exact sequence
\begin{equation*}
   0 \to p^{(0)}_\beta \clopn^{\Gamma_0} \to p^{(0)}_\beta
   \clop^{\Gamma_0} \to \maC_0(S^*M; \End(p^{(0)}_\beta E))^{\Gamma_0}
   \to 0 \, .
\end{equation*}
Moreover,
\begin{equation*}
	\clop^{\Gamma_0} \ \simeq \ \oplus_{\beta \in
          \widehat{\Gamma}_0}\ p^{(0)}_\beta \clop^{\Gamma_0}\, .
\end{equation*}
(Here the direct sum is finite, so there is no need to include the
``$c_0$''-specification like in Corollary \ref{cor.exact}.) Since the actions of
$\Gamma$ and $\Gamma_0$ commute, we can further take the
$\Gamma$-invariants to obtain:
\begin{equation}\label{eq.abelian.one}
   0 \to p^{(0)}_\beta \clopn^{\Gamma} \to p^{(0)}_\beta
   \clop^{\Gamma} \to \maC_0(S^*M; \End(p^{(0)}_\beta E))^{\Gamma} \to
   0 \, .
\end{equation}
Moreover, 
\begin{equation}\label{eq.abelian.two}
   \clop^{\Gamma} \ \simeq \ \oplus_{\beta \in
     \widehat{\Gamma}}\ p^{(0)}_\beta \clop^{\Gamma}\, .
\end{equation}

In particular, we have that

\begin{lemma}\label{lemma.direct.sum.symbol}
Let $\Gamma$ be a finite abelian group and $E \to M$ a
$\Gamma$-equivariant vector bundle over a smooth, compact, connected
manifold $M$ (thus without boundary). Let $\Gamma_0 \subset \Gamma$ be
the minimal isotropy group. We have
\begin{equation*}
   \maC_0(S^*M; \End(E))^{\Gamma} \ \simeq \ \oplus_{\beta \in
     \widehat{\Gamma}_0}\ \maC_0(S^*M; \End(p^{(0)}_\beta E))^{\Gamma}
\end{equation*}
\end{lemma}

\begin{proof} We successively have
\begin{multline*}
  \maC_0(S^*M; \End(E))^{\Gamma} \ \simeq \
  \big( \maC_0(S^*M; \End(E))^{\Gamma_0} \big)^{\Gamma/\Gamma_0}
  \\ 
  \simeq \ \big( \maC_0(S^*M; \End(E)^{\Gamma_0})
  \big)^{\Gamma/\Gamma_0} \ \simeq \
  \oplus_{\beta \in \widehat{\Gamma}_0} \ \big(
  \maC_0(S^*M; \End(p^{(0)}_\beta E)^{\Gamma_0})
  \big)^{\Gamma/\Gamma_0} \\ 
  \simeq \ \oplus_{\beta \in \widehat{\Gamma}_0}
  \ \maC_0(S^*M; \End(p^{(0)}_\beta E))^{\Gamma}\,,
\end{multline*}
where we have used that $\Hom(p^{(0)}_\beta E, p^{(0)}_{\beta'}
E)^{\Gamma_0}$ for $\beta \neq \beta' \in \widehat{\Gamma}_0$.
\end{proof}

Let us record now the following corollary of Proposition
\ref{prop.free.isomorphism}.

\begin{corollary}\label{cor.free.isomorphism}
Let $\Gamma$ be a finite abelian group acting on a smooth, connected
compact manifold $M$ (without boundary). Let $E \to M$ be an
equivariant vector bundle. Assume that minimal isotropy is trivial:
$\Gamma_0 = 1$. Then the map $\maR_M : \maC_0(S^*M; \End(E))^\Gamma
\to \pi_\alpha( \overline{\psi^{0}}(M;
E)^\Gamma)/\pi_\alpha(\overline{\psi^{-1}}(O; E)^\Gamma)$ of Equation
\eqref{eq.def.maRG} is injective, and hence an isomorphism of
algebras.
\end{corollary}

\begin{proof}
By replacing the action $\pi$ of $\Gamma$ on $E$ with $\pi_0 := \pi
\alpha^{-1}$, that is, with, $\pi_0(g) \xi := \alpha^{-1}(g) \pi(g)
\xi$, we can assume that $\alpha = 1$. The action of $\Gamma$ is
moreover free on the dense open subset $M_{(1)} = M_{\Gamma_0}$ of
$M$.  Proposition \ref{prop.free.isomorphism} then allows us to
conclude.
\end{proof}

We now turn to the main result of this section.

\begin{theorem}\label{thm.ker.maRM}
Let $\Gamma$ be a finite abelian group acting on a smooth compact,
connected manifold $M$ (without boundary) and let $E \to M$ be a
$\Gamma$-equivariant vector bundle.  Then the kernel of the morphism
\begin{multline*}
   \maR_M : \oplus_{\beta \in
     \widehat{\Gamma}_0}\ \maC_0(S^*M; \End(p^{(0)}_\beta E))^{\Gamma}
   \
   \simeq \ \maC_0(S^*M; \End(E))^\Gamma \ =: A_M^\Gamma \\
   \to\ \pi_\alpha(
   \overline{\psi^{0}}(O;
   E)^\Gamma)/\pi_\alpha(\overline{\psi^{-1}}(M; E)^\Gamma)
\end{multline*}
of Equation \eqref{eq.def.maRG} is $\oplus_{\beta \in
  \widehat{\Gamma}_0, \beta \neq \alpha'}
\ \maC_0(S^*M; \End(p^{(0)}_{\beta} E))^{\Gamma}$, where $\alpha' :=
\alpha\vert_{\Gamma_0}$. In particular,
$\maR_M(\maC_0(S^*M; \End(E))^\Gamma) \simeq
\maC_0(S^*M; \End(p^{(0)}_{\alpha'} E))^{\Gamma}$.
\end{theorem}

\begin{proof}
It is enough to identify the action of $\maR_M$ on each direct summand
$\maC_0(S^*M; \End(p^{(0)}_\beta E))^{\Gamma}$ of
$\maC_0(S^*M; \End(E))^\Gamma$. We can thus study the action of the
morphism $\maR_M$ one isotypical component $\beta \in
\widehat{\Gamma}_0$ at a time.

The relation between the central projectors of $C^*(\Gamma_0)$ and
$C^*(\Gamma)$ is that $p^{(0)}_\beta = \sum_{\alpha\vert_{\Gamma_0} = \beta}
p_\alpha$, $\beta \in \widehat{\Gamma}_0$. Of course, $p_\alpha
p_{\alpha'} = 0$ if $\alpha \neq \alpha' \in \widehat{\Gamma}$. It
follows that
\begin{equation}
  \pi_{\alpha}(p^{(0)}_\beta P) \seq p_\alpha p^{(0)}_\beta P
  \vert_{L^2(M; E)_{\alpha}} \seq
  \begin{cases}
    \ p_\alpha P\vert_{L^2(M; E)_{\alpha}} & \mbox{ if }
    \ \alpha\vert_{\Gamma_0}=\beta \\
    \quad 0 & \mbox{ otherwise}.
  \end{cases}
\end{equation}
This shows that $\maR_M = 0$ on $\maC_0(S^*M; \End(p^{(0)}_\beta
E))^{\Gamma}$ if $\alpha\vert_{\Gamma_0} \neq \beta$.

On the other hand, for $\alpha\vert_{\Gamma_0} = \beta$, we shall show
that $\maR_M$ is an isomorphism on
$\maC_0(S^*M; \End(p^{(0)}_\beta E))^{\Gamma}$. By replacing the
action $\pi$ of $\Gamma$ on $E$ with $\pi_0 := \alpha^{-1} \pi $, that
is, with, $\pi_0(g) \xi := \alpha^{-1}(g) \pi(g) \xi$, we can assume
that $\Gamma_0$ acts trivially on $E_\beta = p_{\beta}^{(0)} E$ (we
already know that $\Gamma_0$ acts trivially on $M$). We can then
factor the action of $\Gamma$ to an action of $\Gamma/\Gamma_0$, and
thus assume that the minimal isotropy is trivial: $\Gamma_0 = 1$.

After these reductions, the orbit bundle of $M$ is $O_0 := M_{(1)}$,
and the action of $\Gamma$ on $O_0$ is free (and proper since $\Gamma$
is compact). Corollary \ref{cor.free.isomorphism} then shows that
$\maR_M$ is injective on $\maC_0(S^*M; \End(p^{(0)}_\beta
E))^{\Gamma}$.
\end{proof}

We note that the component of $\sigma_m(P)$ in
$\maC_0(S^*M; \End(p^{(0)}_\beta E))^{\Gamma}$ is
$\sigma_m^\alpha(P)$, $\alpha\vert_{\Gamma_0} = \beta$, that
is, the restriction of $\sigma_0^\Gamma(P)$ to $X_{M, E,
  \Gamma}^{\alpha}$. In this regard, we notice that
  \begin{equation}
    \begin{cases}
       \ X_{M, E, \Gamma}^{\alpha} \seq X_{M, E, \Gamma}^{\alpha'} &
       \ \mbox{ if } \alpha\vert_{\Gamma_0} =
       \alpha'\vert_{\Gamma_0}\\
       \ X_{M, E, \Gamma}^{\alpha} \cap X_{M, E, \Gamma}^{\alpha'} = \emptyset &
       \ \mbox{ otherwise. }
    \end{cases}
  \end{equation}
Let us denote $X_{M, E, \Gamma}^{\beta} = X_{M, E, \Gamma}^{\alpha}$
if $\alpha\vert_{\Gamma_0} = \beta$. This gives the {\em disjoint
  union} decomposition
\begin{equation}
   X_{M, \Gamma, E} \seq \bigsqcup_{\beta \in \widehat{\Gamma}_0} X_{M,
     E, \Gamma}^{\beta}\,.
\end{equation}
It would be interesting to establish an analogous relation in the
nonabelian case.

\section{Applications and extensions}
\label{sec.applications}

We now prove the main result of the paper on the characterization of
Fredholm operators and discuss some extensions of our results.

\subsection{Fredholm conditions}
We now turn to the proof of our main result. We assume that $M$ is a
compact smooth manifold. We have the following $\Gamma$--equivariant
version of Atkinson's theorem.

\begin{proposition} \label{prop.Atkinson}
Let $V$ be a unitary $\Gamma$--module and $P$ be a
$\Gamma$--equivariant bounded operator on $V$. We have that $P$ is
Fredholm if, and only if, it is invertible modulo $\maK(V)^\Gamma$, in
which case, we can choose the parametrix (i.e. the inverse modulo the
compacts) to also be $\Gamma$-invariant.
\end{proposition}

\begin{proof}
This follows from the inclusion of
$C^*$-algebras
\begin{equation*}
    \maL(V)^\Gamma/\maK(V)^\Gamma \subset \maL(V)/\maK(V).
\end{equation*}
It is a standard fact that, if $B \subset A$ is an inclusion of unital
$C^*$-algebras, then an element $a \in B$ is invertible in $A$ if, and
only if, it is invertible in $B$ \cite[Proposition
  1.3.10]{Dixmier}. Therefore if $P \in \maL(V)^\Gamma$, then its
projection in $\maL(V)^\Gamma/\maK(V)^\Gamma$ is invertible if, and
only if, it is invertible in the greater algebra $\maL(V)/\maK(V)$. By
Atkinson's theorem, the latter is equivalent to $P$ being Fredholm.
\end{proof}

Since $\pi_\alpha(\maK^\Gamma) = \pi_\alpha(\maK(L^2(M;
  E)_\alpha))^\Gamma$ and $\clopn = \maK := \maK(L^2(M; E)_\alpha)$, we
  obtain the following corollary.

\begin{corollary} \label{cor.Atkinson}
Let $P \in \clop^\Gamma$ and $\alpha \in \hat{\Gamma}$.  We have that
$\pi_\alpha(P)$ is Fredholm on $L^2(M; E)_\alpha$ if, and only if,
$\pi_\alpha(P)$ is invertible modulo $\pi_\alpha(\maK^\Gamma)$ in
$\pi_\alpha(\clop^\Gamma)$.
\end{corollary} 

We are now in a position to prove the main result of this paper,
Theorem \ref{thm.main1}.

\begin{proof}[Proof ot Theorem \ref{thm.main1}]
Lemma \ref{lemma.red.zero} implies that we may assume $P \in
\overline{\psi^0}(M;E)^\Gamma$. Corollary \ref{cor.Atkinson} then
states that $\pi_\alpha(P)$ is Fredholm if, and only if, the image of
its principal symbol $\sigma_0(P)$ is invertible in the quotient
algebra
\begin{equation*}
   \maR_M(A_M^\Gamma) =
   \pi_\alpha(\overline{\psi^0}(M;E)^\Gamma)/\pi_\alpha(\maK^\Gamma).
\end{equation*}

We have shown the isomorphism $\Prim (A_M^\Gamma) \simeq X_{M, E,
  \Gamma}/\Gamma$ in Equation \eqref{eq.bijection.Omega}. Now Theorem
\ref{thm.ker.maRM} and the discussion following it identify the
primitive spectrum of $\maR_M(A_M^\Gamma)$, which is a closed subset
of $\Prim(A_M^\Gamma)$, with the set $X_{M, E,
  \Gamma}^\alpha/\Gamma$. Recall that
\begin{equation*}
    X_{M, E, \Gamma}^\alpha = \{(\xi,\rho) \in
    T^*M\setminus\{0\}\times\widehat{\Gamma}_\xi \mid \rho_{|\Gamma_0}
    = \alpha_{|\Gamma_0} \},
\end{equation*}
as defined in the introduction.  

Therefore $\maR_M(\sigma(P))$ is invertible if, and only if, the
endomorphism $\pi_{\xi,\rho}(\sigma(P))$ is invertible for all
$(\xi,\rho) \in X_{M, E, \Gamma}^\alpha$, i.e.\ if and only if $P$ is
$\alpha$-elliptic.
\end{proof}

\subsection{Boundary value problems}
In this subsection, we very briefly indicate an application to
  mixed boundary value problems. Let $M$ be a smooth compact manifold
with boundary and choose a tubular neighborhood $U \simeq [0,1) \times
  \pa M$ of the boundary. Let $M^d$ be the \emph{double} of $M$ along
  $\pa M$: as a topological space, the space $M^d$ is the quotient of
  $M\times\{-1,1\}$ by the subspace $\pa M \times \{-1,1\}$. We shall
  denote $M_\pm \ede M\times\{\pm 1\}$. On $M^d$ we consider the
    smooth structure such that
\begin{enumerate}
  \item the projections $p_{\pm} : M_\pm\to M$ are smooth maps, and
  \item the map $U^d \simeq (0,1) \times \pa M$ is smooth.
\end{enumerate}
Thus the smooth structure on $M^d$ thus depends on our choice of
tubular neighborhood.  For any $x = (x',i) \in M^d$, we denote by $-x$
its symmetrical counterpart, i.e.\ $-x = (x',-i)$. Then the map $x
\mapsto -x$ gives a natural smooth action of $\ZZ_2$ on $M^d$.

If $E \to M$ is a smooth vector bundle, then we define $E^d \to M^d$
as the smooth vector bundle obtained by gluing two copies of $E$ on
$M_+$ and $M_-$ along $\pa M$. Then the $\ZZ_2$-action on $M^d$
extends to an action on $E^d$, which maps an element $v \in E^d_x$ to
its copy in $E^d_{-x}$, for any $x \in M^d$.

We generalize this construction to the case when we have
a disjoint union decomposition of the boundary $\pa M = \pa_D M \cup
\pa_N M$ into two disjoint, closed and open subsets. Then we double
first with respect to the ``Dirichlet'' part of the boundary $\pa_D M$
and then with respect to the ``Neumann'' part of the boundary $\pa_N
M$. We obtain accordingly an action of $\ZZ_2^2$ on the resulting
manifold $M^{dd}$. We let this group act on the resulting vector
bundle $E^{dd}$ such that the action of the first component of $\ZZ_2$
is twisted (i.e. tensored) with its only non-trival character, namely
$-1$. We have the following standard lemma.

\begin{lemma} \label{lm.restriction.even.odd}
The restriction map $r_+ :\maC^\infty(M^{dd};E^{dd}) \to
  \maC^\infty(M_+;E)$ induces a isomorphisms
  \begin{enumerate}
    \item $L^2(M^{dd};E^{dd})^{\ZZ_2^2} \, \simeq \, L^2(M;E)$,
    \item $H^2(M^{dd};E^{dd})^{\ZZ_2^2} \, \simeq \, H^2(M;E) \,
      \cap\, \{\, u\vert_{\pa_D M} = 0 ,\ \pa_\nu u \vert_{\pa_N M} =
      0 \, \}$.
  \end{enumerate}
\end{lemma}

An order-$2$, $\ZZ_2^2$-invariant pseudodifferential operator $P$ on
$M^{dd}$ will map invariant sections to invariant sections; this means
that we consider the case $\alpha = 1$ in Theorem
\ref{thm.main1}. Because the action of $\ZZ_2^2$ is free on a dense
subset of $M^{dd}$, Theorem \ref{thm.main1} implies that $P$ is
Fredholm from $H^2(M^{dd};E^{dd})^{\ZZ_2^2}$ to
$L^2(M^{dd};E^{dd})^{\ZZ_2^2}$ if, and only if, it is elliptic. This
then yields Fredholm conditions for the restriction of $P$ to $M$,
with mixed Dirichlet/Neumann boundary conditions on $\pa_D M$ and
$\pa_N M$.

\subsection{The case of non-discrete groups}
If $\Gamma$ is not discrete, then it is enough for our operators to be
{\em transversally} elliptic.  Indeed, let us assume that $M$ is a
compact smooth manifold and that $\Gamma$ is a compact Lie group
acting on $M$. Denote by $\mathfrak{g}$ the Lie algebra of
$\Gamma$. Recall that any $X \in \mathfrak{g}$ defines as usual the
vector field $X^*_M$ given by $X^*_M(m)=\frac{d}{dt}_{|_{t=0}}
e^{tX}\cdot m$.  Let first introduce the $\Gamma$-transversal space
$$T^*_\Gamma M := \{\alpha \in T^*M\ |\ \alpha(X^*_M(\pi(\alpha)))=0,
\forall X\in \mathfrak{g}\}.$$ A $\Gamma$-invariant classical
pseudodifferential operator $P$ of order $m$ is said
{\em $\Gamma$-transversally elliptic} if its principal symbol is invertible
on $T^*_\Gamma M\setminus \{0\}$.  Let $P \in \psi^m(M; E_0,E_1)$ be
$\Gamma$-transversally elliptic. Recall the now classical result of
Atiyah and Singer \cite[Corollary 2.5]{atiyahGelliptic} 

\begin{theorem}
Assume $P$ is $\Gamma$-transversaly elliptic. Then for every
irreducible representation $\alpha\in \widehat{\Gamma}$,
  $$\pi_\alpha(P) : H^s(M; E_0)_\alpha \, \to \, H^{s-m}(M;
  E_1)_\alpha,$$ is Fredholm.
\end{theorem}

Note that this implies that Theorem \ref{thm.main1} is not true anymore for
$\Gamma$-transversally elliptic operators if $\Gamma$ is non-discrete.

\end{document}